\documentclass[letterpaper,10pt,reqno]{amsart}
\usepackage{indentfirst} 
\usepackage{amssymb}
\usepackage{mathtools} 
\usepackage{mathabx} 
\usepackage{amsthm}
\usepackage{thmtools}
\usepackage{enumitem} 
\usepackage[colorlinks=true]{hyperref} 
\usepackage[dvipsnames]{xcolor} 
\usepackage{ifthen}
\usepackage{tikz}
\usetikzlibrary{decorations.pathreplacing}
\usepackage{tikz-cd}
\usepackage[left=3.2cm, right=3.2cm, bottom=3.4cm]{geometry}
  
\makeatletter

\def\MRbibitem{\@ifnextchar[\my@lbibitem\my@bibitem}

\def\mybiblabel#1#2{\@biblabel{{\hyperref{http://www.ams.org/mathscinet-getitem?mr=#1}{}{}{#2}}}}

\def\myhyperanchor#1{\Hy@raisedlink{\hyper@anchorstart{cite.#1}\hyper@anchorend}}

\def\my@lbibitem[#1]#2#3#4\par{%
  \item[\mybiblabel{#2}{#1}\myhyperanchor{#3}\hfill]#4%
  \@ifundefined{ifbackrefparscan}{}{\BR@backref{#3}}%
  \if@filesw{\let\protect\noexpand\immediate
    \write\@auxout{\string\bibcite{#3}{#1}}}\fi\ignorespaces%
}

\def\my@bibitem#1#2#3\par{%
  \refstepcounter\@listctr
  \item[\mybiblabel{#1}{\the\value\@listctr}\myhyperanchor{#2}\hfill]#3%
  \@ifundefined{ifbackrefparscan}{}{\BR@backref{#2}}%
  \if@filesw\immediate\write\@auxout
    {\string\bibcite{#2}{\the\value\@listctr}}\fi\ignorespaces%
}

\makeatother

\DeclareFontFamily{U} {MnSymbolA}{}
\DeclareFontShape{U}{MnSymbolA}{m}{n}{
   <-6> MnSymbolA5
   <6-7> MnSymbolA6
   <7-8> MnSymbolA7
   <8-9> MnSymbolA8
   <9-10> MnSymbolA9
   <10-12> MnSymbolA10
   <12-> MnSymbolA12}{}
\DeclareFontShape{U}{MnSymbolA}{b}{n}{
   <-6> MnSymbolA-Bold5
   <6-7> MnSymbolA-Bold6
   <7-8> MnSymbolA-Bold7
   <8-9> MnSymbolA-Bold8
   <9-10> MnSymbolA-Bold9
   <10-12> MnSymbolA-Bold10
   <12-> MnSymbolA-Bold12}{}
\DeclareSymbolFont{MnSyA} {U} {MnSymbolA}{m}{n}
 \DeclareFontFamily{U} {MnSymbolC}{}
\DeclareFontShape{U}{MnSymbolC}{m}{n}{
  <-6> MnSymbolC5
  <6-7> MnSymbolC6
  <7-8> MnSymbolC7
  <8-9> MnSymbolC8
  <9-10> MnSymbolC9
  <10-12> MnSymbolC10
  <12-> MnSymbolC12}{}
\DeclareFontShape{U}{MnSymbolC}{b}{n}{
  <-6> MnSymbolC-Bold5
  <6-7> MnSymbolC-Bold6
  <7-8> MnSymbolC-Bold7
  <8-9> MnSymbolC-Bold8
  <9-10> MnSymbolC-Bold9
  <10-12> MnSymbolC-Bold10
  <12-> MnSymbolC-Bold12}{}
\DeclareSymbolFont{MnSyC} {U} {MnSymbolC}{m}{n}

\DeclareMathSymbol{\top}{\mathord}{MnSyA}{219} 
\DeclareMathSymbol{\plus}{\mathord}{MnSyC}{20} 

\declaretheorem[numberwithin=section]{theorem} 
\declaretheorem[sibling=theorem]{proposition} 
\declaretheorem[sibling=theorem]{lemma}
\declaretheorem[sibling=theorem]{corollary}

\declaretheorem[sibling=theorem, style=remark]{claim}
\declaretheorem[sibling=theorem, style=definition]{example}
\declaretheorem[sibling=theorem, style=definition]{definition}
\declaretheorem[sibling=theorem]{remark}

\hypersetup{bookmarksdepth = 3} 
\numberwithin{equation}{section}     

\setlist[enumerate,1]{label={\upshape(\alph*)},ref=\alph*}
\setlist[enumerate,2]{label={\upshape(\arabic*)},ref=\arabic*}
\providecommand{\norm}[1]{\lVert#1\rVert}
\newcommand{\M}{\mathcal{M}}
\newcommand{\R}{\mathbb{R}}
\newcommand{\Z}{\mathbb{Z}}
\newcommand{\N}{\mathbb{N}}

\newcommand{\cK}{\mathcal{K}}

\renewcommand{\d}{~{\rm d}}
\renewcommand{\epsilon}{\varepsilon}

\def\R{{\mathbb R}}

\def\N{{\mathbb N}}
\def\Z{{\mathbb Z}}

\def\B{{\mathcal B}}

\def\F{{\mathcal F}}

\def\M{{\mathcal M}}

\def\supp{\mbox{\rm supp}}

\def\le{\leqslant}
\def\ge{\geqslant}

\def\F{\mathcal{F}}
\def\M{\mathcal{M}}

\newcommand{\vertiii}[1]{{\left\vert\kern-0.25ex\left\vert\kern-0.25ex\left\vert #1 
    \right\vert\kern-0.25ex\right\vert\kern-0.25ex\right\vert}}
\newcommand{\invertiii}[1]{{\vert\kern-0.25ex\vert\kern-0.25ex\vert #1 
    \vert\kern-0.25ex\vert\kern-0.25ex\vert}}

\begin{document}

\title{Pressure at infinity  and Large Deviations}
\date{\today}

\subjclass[2010]{37D35, 37A10, 37A35, 60F10}

\begin{thanks}
{G.I.\ was partially supported  by Proyecto Fondecyt Regular 1230100. A.V.\ was partially supported  by Proyecto Fondecyt Regular 1250928. We thank Andr\'es Navas for fruitful discussions on the concentration-compactness principle. This work was started during the second author's visit to Pontificia Universidad Cat\'olica de Chile (UC). He wishes to thank UC for the hospitality during the visit.}
\end{thanks}

\author[G.~Iommi]{Godofredo Iommi}
\address{Facultad de Matem\'aticas,
Pontificia Universidad Cat\'olica de Chile (UC), Avenida Vicu\~na Mackenna 4860, Santiago, Chile}
\email{\href{godofredo.iommi@gmail.com}{godofredo.iommi@gmail.com}}
\urladdr{\url{http://www.mat.uc.cl/~giommi/}}

 \author[T.~Jordan]{Thomas Jordan} 
\address{School of Mathematics, Fry Building, University of Bristol, Bristol, BS8 !UG, U.K.}
\email{\href{thomas.jordan@bristol.ac.uk}{thomas.jordan@bristol.ac.uk}}
\urladdr{\url{https://people.maths.bris.ac.uk/~matmj}}
 \author[A.~Velozo]{Anibal Velozo}  \address{Facultad de Matem\'aticas,
Pontificia Universidad Cat\'olica de Chile (UC), Avenida Vicu\~na Mackenna 4860, Santiago, Chile}
\email{\href{apvelozo@uc.cl}{apvelozo@uc.cl}}
\urladdr{\href{https://sites.google.com/view/apvelozo}{https://sites.google.com/view/apvelozo}}

\begin{abstract}
We establish level-1 large deviation principles for continuous time suspension flows over countable Markov shifts, a class of non compact dynamical systems. Our results both quantify deviations of Birkhoff averages from their typical values and, importantly, provide a large deviation principle for orbits that at a finite time are close to escaping to infinity. To overcome the challenges posed by non compactness, we introduce the notion of topological pressure at infinity and prove a variational principle linking it to its measure theoretic counterpart. This framework allows the study of a natural class of potentials, known as strongly positive recurrent, which exhibit robust thermodynamic properties. Moreover, we show that compactly supported ergodic measures are pressure dense in this setting. Applications to pressure gaps and examples are given. 
\end{abstract}

\maketitle

\section{Introduction}

 In dynamical systems, large deviation theory provides a quantitative description of how often time averages of observables differ significantly from their typical values. For a system endowed with an invariant probability measure, Birkhoff’s ergodic theorem ensures that these time averages converge almost surely to the corresponding space average. Large deviation principles refine this result by estimating the exponential rate at which the probability of such deviations decays. In the level-1 setting, the focus is on deviations of Birkhoff averages of a single observable, leading to rate functions that are typically convex and, in many expanding or hyperbolic systems, can be expressed as Legendre transforms of pressure functions arising from thermodynamic formalism.
 
Beginning in the early 1990s, several fundamental large deviation results were established in the context of dynamical systems. In particular, the works of Kifer~\cite{k} and L.S.~Young~\cite{y} proved level-1 large deviation principles for systems exhibiting sufficient hyperbolicity and for observables with adequate regularity. Since then, these results have been extended to include broader classes of systems and functions. 

In this article, we study continuous time dynamical systems defined on non-compact phase spaces, more precisely suspension flows built over countable Markov shifts. 
The lack of compactness of the phase space constitutes a major obstacle to the  application of many of the standard techniques in the area. Over the past few years, countable Markov partitions have been constructed for a broad class of continuous time dynamical systems. In these settings, the symbolic model is given by a suspension flow over a countable Markov shift. A major recent breakthrough, due to Buzzi, Crovisier, and Lima \cite{bcl}, extends earlier work of Lima and Sarig \cite{ls} by providing symbolic dynamics for three-dimensional flows with positive speed. Earlier constructions of Markov partitions for Axiom A flows on compact manifolds were obtained in the 1970s by Bowen \cite{bow} and Ratner \cite{ra}. The results in \cite{bcl} significantly generalize these classical constructions to the non-uniformly hyperbolic setting.

In problems formulated on non-compact spaces, it is often both natural and fruitful to analyze the behaviour of the system \emph{at infinity}.   A classical and highly influential development of this viewpoint was carried out by P.-L. Lions in the mid-1980s in his study of minimization problems on $\R^n$ \cite{l1,l2,l3,l4}. He established the celebrated \emph{concentration–compactness principle}, which describes the possible mechanisms by which convergence of approximate solutions can fail. In particular, it characterizes phenomena such as the escape of mass to infinity that lead to loss of compactness. The principle can be adapted to treat other quantities, such as energy, and its heuristic is outlined in \cite[Section I]{l1}. Numerous applications have since appeared in the theory of PDEs and the calculus of variations.

Related ideas have also appeared in dynamical systems and homogeneous dynamics. 
For instance, in their ergodic-theoretic proof of Duke’s theorem on the equidistribution of closed geodesics on the modular surface, Einsiedler, Lindenstrauss, Michel, and Venkatesh~\cite{elmv} connected the loss of mass of invariant measures to the entropy carried to infinity. 
Further developments along these lines were obtained in~\cite{ek} and~\cite{ekp}.

Our focus lies on thermodynamic formalism for dynamical systems with non-compact phase spaces. 
For countable Markov shifts, several notions designed to quantify the complexity of the dynamics at infinity have been introduced and studied in~\cite{bu,itv,ru1,ru2}. 
In particular, both measure theoretic and topological notions of \emph{entropy at infinity} have been defined and shown to be closely related. 
These ideas have subsequently been extended to the notion of \emph{pressure at infinity}; see~\cite{irv1, rs,v}.  In the continuous time setting, analogous concepts have been developed for geodesic flows on non-compact manifolds, where entropy and pressure at infinity have also been investigated, see for example \cite{irv1, rv, gst,v2}.

Topologies other than the weak$^*$ topology are needed to capture thermodynamic behaviour at infinity. We consider the cylinder topology for both discrete and continuous time systems, see Sections \ref{topologies} and \ref{topo-flow}. In this topology, sequences of probability measures may converge to sub-probability measures, thereby allowing the escape of mass to be quantified. We emphasize that, unlike in the manifold setting, the phase space in this symbolic context does not need to be locally compact. In particular, the space of compactly supported continuous functions may be trivial. In consequence, notions of convergence defined through such test functions, such as vague convergence, are not useful in this setting.

For suspension flows built over countable Markov shifts, the notion of entropy at infinity was introduced and studied in~\cite{irv1}, while the measure-theoretic pressure at infinity was analyzed in~\cite{v,irv2}. In Section~\ref{pres_inf_flow}, we extend the concept of \emph{topological pressure at infinity} for suspension flows to a broader setting. In Theorem~\ref{prin_var_flow_infty}, we relate this notion to its measure-theoretic counterpart and establish a variational principle at infinity. We further investigate several properties of this pressure. In particular, we show that its derivatives can be described in terms of sequences of invariant measures whose total mass escapes the system and which play a role analogous to that of equilibrium measures for the classical pressure (see Proposition~\ref{der.flow}).
We establish a connection between the pressure at infinity and the upper semicontinuity of the pressure map. Specifically, we demonstrate that the pressure at infinity precisely quantifies the extent to which upper semicontinuity fails (see Theorem~\ref{thm:usc_sus}).
We also consider a class of functions, referred to as \emph{strongly positive recurrent} (SPR), for which the pressure at infinity is strictly smaller than the full pressure. This class of functions has unique equilibrium measures that exhibit robust ergodic properties.

For this class of continuous time systems and SPR potentials, we establish a level-1 large deviation principle, see
Theorem~\ref{thm:ldp}. This result extends several large deviation theorems
from the compact setting to the present non-compact symbolic framework, see
for example \cite[Theorems~A and~C]{ekw}. Related questions have also been
studied by different methods, most notably by Takahashi~\cite{ta1,ta2} for
discrete-time countable Markov shifts; see also \cite{kpw}. Further results concerning suspension
flows  appear in \cite{ab}.

Let $(Y,\Phi)$ be a suspension flow over a countable Markov shift. Theorem~\ref{thm:ldp} has two complementary components.  The first concern
ordinary deviations of Birkhoff averages from their typical values. More
precisely, given an observable $g:Y\to\R$, one studies the exponential asymptotics
of the set of orbit segments for which
\[
   \frac{1}{T}\int_0^T g(\Phi_t y)\,dt
\]
differs from the value prescribed by the equilibrium state. Some care is
needed at this point: even when the base transformation is mixing, the
associated suspension flow need not be mixing. Therefore, the large
deviation statement must be formulated on a suitable family of recurrent
orbit segments. This is the role of the set $R_T(m)$. 

The rate function governing these large deviations is expressed in terms of
pressure. Let $f,g\colon Y\to\mathbb R$ be continuous functions defined on the suspension space $Y$ that induce regular functions, $\Delta_f$ and $\Delta_g$, on the base (see Section \ref{ssec:meas AK} for precise definitions). Assume that the potential
$f$ has zero pressure  and that $g$ is bounded. For $\gamma\in\mathbb R$, we define
\[
   I^{f,g}(\gamma)
   =
   \inf_{q\geq 0} P^\Phi\bigl(f+q(g-\gamma)\bigr),
\]
where $P^\Phi$ denotes the pressure on the flow.  The second component of Theorem~\ref{thm:ldp} is specific to the non-compact
setting. It describes large deviations caused not by atypical behaviour in
the recurrent part of the phase space, but by orbit segments which, on
average and at a suitable scale, escape to infinity. The corresponding rate function is defined by
replacing the usual pressure with the pressure at infinity. Namely, for
$\gamma\in\mathbb R$, we set
\[
   I^{f,g}_\infty(\gamma)
   =
   \inf_{q\geq 0}
   P^\Phi_\infty\bigl(f+q(g-\gamma)\bigr).
\]
Here $P^\Phi_\infty$ denotes the pressure carried by sequences of invariant
measures which lose mass in the cylinder topology (see Section \ref{pres_inf_flow} for details and properties). Accordingly,
$I^{f,g}_\infty(\gamma)$ quantifies the exponential contribution of orbit
segments whose empirical measures escape to infinity while their time
averages remain constrained by the observable $g$. In Section \ref{sec:rate} precise definitions and properties of the rate functions are given.

Theorem~\ref{thm:ldp} therefore separates two different mechanisms for large
deviations: deviations produced inside the recurrent part of the system,
captured by $I^{f,g}$, and deviations produced by escape of mass, captured by
$I^{f,g}_\infty$. 

In order to make the result more precise consider a countable Markov shift $(\Sigma, \sigma)$ defined on the  alphabet $\N$. Let $(Y,\Phi)$ be the associated suspension flow with roof function $\tau:\Sigma\to(0,\infty)$ that is bounded away from zero. Let $Y_0= \{ (x,t)\in \Sigma  \times \R \colon 0 \le t <\tau(x)\}$ be the suspension space without identification of the end points. Fix $m, Q, M \in\N$, set $K_m= \bigcup_{i=1}^m [i]$ and $\widetilde{K}_m=(K_m\times\R)\cap Y_0$. Define $$R_T(m)=\widetilde{K}_m\cap \Phi^{-T} \big(\widetilde{K}_m\big),$$
and 
$$R_T(m;Q,M)=\left\{(x,t)\in R_T(m): \frac{1}{T}\int_0^T \chi_{\widetilde{K}_M}(\Phi_s((x,t)))\,\d s <\frac{1}{Q}\right\},$$
where $\chi_A$ is the characteristic function of $A\subset Y.$ We now state our main result. Precise definitions of the objects involved will be given throughout the text.

\begin{theorem}\label{thm:ldp}

Let $(\Sigma,\sigma)$ be a topologically mixing countable Markov shift and let $\tau:\Sigma\to\R$ be a roof function that is bounded away from zero, has summable variations and satisfies 
$\sup_{x\in[b]}\tau(x)< \infty$ for every $b\in\N$.
Denote by $(Y,\Phi)$ the associated suspension semiflow. 

Let  $f:Y\to\R$ be a continuous and such that $\Delta_f$ has summable variations and $\sup_{x\in[b]}|\Delta_f(x)|< \infty$ for every $b\in\N$. Suppose that $P^\Phi(f)=0$  and that $f$ admits an equilibrium state $\nu_f$. Finally, let $g:Y\to\R$ be continuous and bounded with  $\Delta_g$ of summable variations. 

Then, there exists $m\in\N$ such that:

\begin{enumerate}
\item  Suppose that $\gamma<\beta(g)$. Then,
$$\lim_{T\to\infty}\frac{1}{T}\log\nu_f\left( \bigg\{(x,t)\in R_T(m): \frac{1}{T}\int_0^T g(\Phi_s((x,t))\,d s> \gamma \bigg\}\right)=I^{f,g}(\gamma).$$
If in addition $f$ is SPR, $s_\infty(f)<1$, and $\gamma\in (\int g \,\d\nu_f,\beta(g))$, then  $I^{f,g}(\gamma)<0$. 

\item  Suppose that $\gamma<\beta_\infty(g)$, and that $I_\infty(\gamma)>-\infty$. Let $ Q,M\in \N$. Then,
$$I^{f,g}_\infty(\gamma)\le \liminf_{T\to\infty}\frac{1}{T}\log\nu_f \bigg(\bigg\{(x,t)\in R_T(m; Q,M): \frac{1}{T}\int_0^T g(\Phi_s((x,t))\,\d s> \gamma \bigg\}\bigg),$$
and
$$\limsup_{T\to\infty}\frac{1}{T}\log\nu_f\bigg(\bigg\{(x,t)\in R_T(m; Q,M): \frac{1}{T}\int_0^T g(\Phi_s((x,t))\,\d s> \gamma \bigg\}\bigg) \le I^{f,g}_\infty(\gamma)+\delta(Q,M),$$
where $\lim_{Q,M\to\infty} \delta(Q,M)=0$. 
If in addition $f$ is SPR, and $\gamma<\beta_\infty(g)$, then $I^{f,g}_\infty(\gamma)<0$.

\item Suppose that $\gamma<\beta_\infty(g)$, and that $I^{f,g}_\infty(\gamma)=-\infty$. Then,
$$\lim_{Q,M\to\infty}\limsup_{T\to\infty}\frac{1}{T}\log\nu_f\bigg(\bigg\{(x,t)\in R_T(m; Q,M): \frac{1}{T}\int_0^T g(\Phi_s((x,t))\,\d s> \gamma \bigg\}\bigg)  =-\infty$$

\end{enumerate}
\end{theorem}

Note that the case of averages strictly smaller than $\gamma$ follows by replacing $g$ with $-g$ which yields the corresponding large deviation result. Moreover, the upper bound holds for every $m \in \N$
and therefore we obtain the following.

\begin{corollary}\label{cor:version2}
Consider the assumptions of Theorem \ref{thm:ldp}. Suppose that $f$ is SPR. Then for every $m\in\N$ the following holds:
\begin{enumerate}
\item Assume that $s_\infty(f)<1$. For every $\epsilon>0$ we have that 
\begin{equation*}
\limsup_{T\to\infty}\frac{1}{T}\log\nu_f\left( \bigg\{(x,t)\in R_T(m): \left|\frac{1}{T}\int_0^Tg(\Phi_s(x,t)) \,\d s-\int g \, \d\nu_f \right|> \epsilon \bigg\}\right)<0.
\end{equation*}
\item If $Q$ and $M$ are sufficiently large then
$$\limsup_{T\to\infty}\frac{1}{T}\log \nu_f\left( R_T(m; Q,M) \right)<0.$$
\end{enumerate}
\end{corollary}

In the case that $f$ has an equilibrium state but it is not SPR, then the strict negativity of the rate function in part (a) of Theorem \ref{thm:ldp} does not necessarily hold (see Remark \ref{rem:sharp}).

For the precise definition of $s_\infty(f)$ see formula (\ref{s_infty}). We generally assume $s_\infty(f)<1$ to ensure that Theorem \ref{thm:spr_deri} holds. However, this condition can sometimes be relaxed.
For example, if the suspension flow has finite entropy (see \cite[Lemma 4.12]{irv2}), or in the case of the constant roof function where the functions $f,g$ are weakly H\"older (see \cite[Theorem 3.2]{rs} and the remark right above). Furthermore, under the assumption that the base countable Markov shift is the full shift in Theorem \ref{cor:full}, we obtain conditions under which the assumption on $R_T(m)$ can be relaxed. 

Applications and examples of our main theorem are provided in Section \ref{sec:ex}. For instance, we prove that  a strongly positive regular function either has a Markov measure as an equilibrium measure or the free energy restricted to Markov measures is bounded away from the pressure. A version of this result for suspension flows is also given. These results generalize the well known Hausdorff dimension gap theorem obtained by Kifer, Peres and Weiss in the continued fraction setting \cite{kpw}. 

Beyond their role in the proofs of our main results, we establish, as an independent contribution, that compactly supported ergodic measures are \emph{pressure dense} for both countable Markov shifts (Theorem \ref{pressure_density}) and suspension flows (Theorem \ref{pressure_density_susp}). Results of this type, previously obtained in compact settings, have proved useful in a wide range of situations (see, for example, \cite[Theorem B]{ekw} or \cite[Section 9.4]{gk} for entropy density). \\

{\bf Organization of the article.} In Section~\ref{sec:cms}, we review the thermodynamic formalism for
countable Markov shifts, both in its classical form and at infinity, and we
discuss the cylinder topology in this setting. In
Theorem~\ref{pressure_density}, we establish a pressure density result for
compactly supported ergodic measures.
Section~\ref{sec:susp} is devoted to suspension flows over countable Markov
shifts. There we recall the classical thermodynamic formalism for such flows,
study the corresponding cylinder topology, and prove a pressure density
result analogous to the discrete time case, see
Theorem~\ref{pressure_density_susp}. In Section~\ref{pres_inf_flow}, we
introduce the topological pressure at infinity for suspension flows and prove
that it coincides with its measure theoretic counterpart, see
Theorem~\ref{prin_var_flow_infty}. We also study several properties of the
pressure at infinity and the corresponding notion of strong positive
recurrence. We prove some compactness result for the space of invariant sub-probability measures, upper semi-continuity of the pressure and calculate its derivative.
Section~\ref{sec:rate} is devoted to the definition and analysis of the rate
functions, both classical and at infinity, that are used in the large
deviation results. In Section~\ref{sec:ldp}, we prove our main large
deviation theorem, Theorem~\ref{thm:ldp}. Finally, in
Section~\ref{sec:ex}, we apply our results and present illustrative examples. Several
additional results are obtained under the assumption that the base countable
Markov shift of the suspension flow is a full shift, this assumption allows us
to strengthen some of the preceding conclusions. The article concludes with
examples in which the base dynamics do not satisfy the big images and
preimages property.

\section{Countable Markov shifts}   \label{sec:cms}
In this section we not only  review the thermodynamic formalism for countable Markov shifts (CMS) and the associated space of invariant probability and sub-probability measures, but also establish some new results regarding pressure density (see Proposition  \ref{pressure_density}) and the derivatives of the pressure and the pressure at infinity (see Proposition \ref{der}).

Let $M$ be an $\mathbb{N} \times \mathbb{N}$ infinite matrix with entries of either $0$ or $1$. Consider the following associated space:
\begin{equation*}
 \Sigma=\left\{ (x_1, x_2, \dots) \in \N^{\N}: M(x_i,x_{i+1})=1 \text{ for every } i \in \N \right\}.
\end{equation*} 
Endow $\N$ with the discrete topology and $\N^{\N}$ with the product topology and on $\Sigma$ consider the induced topology.   The \emph{shift map} $\sigma:\Sigma \to \Sigma$ is defined by $\sigma(x)=(x_2,x_3,\ldots)$, where $x=(x_1, x_2, \dots ) \in \Sigma$. The dynamical system $(\Sigma,\sigma)$ is called  a  \emph{countable Markov shift} (CMS). Note that the matrix $M$ can be represented by a directed graph whose vertices are labeled by $\mathbb{N}$ and whose edges are determined by the entries of $M$. 
An \emph{admissible word} of length $N$ is a string $a_1a_2\ldots a_{N}$ of letters in $\N$ such that $M(a_i,a_{i+1})=1$ for every $i\in\{1,\ldots,N-1\}$.
We define an initial metric on $\Sigma$ by setting 
$$d({a},{b})=\left(\frac{1}{2}\right)^{|{a}\wedge {b}|}$$
where 
$$|{a}\wedge {b}|=\inf\{n-1:a_n\neq b_n\}.$$
A \emph{cylinder} of length $N$ is a set of the form 
\begin{equation*}
[a_1,\ldots,a_{N}]= \left\{ (x_1,x_2,\ldots)\in \Sigma :  x_i=a_i  \text{ for } 1 \le i \le N \right\}.
\end{equation*} 
Denote by $|\cdot |$ the length of a cylinder, that is, $|[a_1,\ldots,a_{N}]|=N$.
Cylinder sets are balls in the metric $d$ and form a basis for the topology on $\Sigma$. 
The system $(\Sigma,\sigma)$ is \emph{topologically transitive}  if for every $a, b \in \N$, there exists an admissible word that starts with $a$ and ends with $b$.  We say that $(\Sigma,\sigma)$ is \emph{topologically mixing} if for every pair $a, b \in \N$, there exists a number $N(a,b)$ such that for every $n \ge N(a,b)$, there exists an admissible word of length $n$ that starts with $a$ and ends with $b$. For a function $\phi:\Sigma\to\R$ we define the Birkhoff sum $S_n\phi(x)=\sum_{k=0}^{n-1}\phi(\sigma^kx).$

\subsection{Topologies in the space of shift invariant measures} \label{topologies}
Denote by $\M_{\sigma}$ the space of $\sigma$-invariant Borel probability measures on $\Sigma$, and by $\mathcal{M}_{\le 1}(\sigma)$
the space of $\sigma$-invariant Borel sub-probability measures on $\Sigma$. Denote by $\mathrm{C}_b(\Sigma)$ the space of bounded and continuous real functions. 

We consider two distinct notions of convergence for measures. A sequence of measures $(\mu_n)_n$ in  $\M_{\sigma}$ converges weak$^*$ to a measure $\mu$  if for every $\phi \in \mathrm{C}_b(\Sigma)$ we have 
\begin{equation*}
\lim_{n \to \infty} \int \phi \, d \mu_n = \int \phi \, d \mu.
\end{equation*}
This notion of convergence defines a metrizable topology on $\M_{\sigma}$, known as the weak$^*$ topology. Probability measures are preserved by the weak$^*$ topology. If $\Sigma$ is compact, the space $\M_{\sigma}$ is also compact. 
However, if $(\Sigma, \sigma)$ is non-compact and transitive, then $\M_{\sigma}$ is non-compact. 
In \cite{iv} the following notion of convergence of measures was considered. Let $(\mu_n)_n$ and $\mu$ be measures in $\mathcal{M}_{\le 1}(\sigma)$. We say that $(\mu_n)_n$ converges on {cylinders}  to $\mu$ if for every cylinder $C \subseteq \Sigma$ we have
\begin{equation*}
\lim_{n \to \infty}  \mu_n(C)=  \mu(C).
\end{equation*}
This notion of convergence defines a metrizable topology on $\mathcal{M}_{\le 1}(\sigma)$ called the cylinder topology. A sequence of invariant probability measures can converge in the cylinder topology to a sub-probability measure. That is, this topology captures the escape of mass phenomena that naturally occur in non-compact spaces. However, when there is no escape of mass, the cylinder and the weak$^*$ topologies coincide (see \cite[Lemma 3.17]{iv}). It turns out that if $(\Sigma, \sigma)$ satisfies a combinatorial condition (called the $\F-$property, see \cite[Definition 4.9]{iv}) which is fulfilled by every finite entropy CMS and also for every locally compact CMS, then $\mathcal{M}_{\le 1}(\sigma)$ is compact with respect to the cylinder topology (see \cite[Theorem 1.2]{iv}).

\subsection{Pressure for CMS}

Let $(\Sigma,\sigma)$ be a transitive countable Markov shift with alphabet $\N$. The \emph{$n$-th variation} of $\phi:\Sigma \to \R$ is defined by
\[ V_{n}(\phi):= \sup \{| \phi(x)-  \phi(y)| : x,y \in \Sigma, x_{i}=y_{i}, 1 \leq i \leq n \}. \]
We say that $ \phi$ is of \emph{summable variations} if $\sum_{n=2}^{\infty} V_n(\phi)< \infty$. We say that it is  \emph{locally H\"older} (with parameter $\theta$) if there exists $\theta \in (0,1)$ such that for all $n \geq 1$ we have  $V_{n}( \phi) \leq O( \theta^{n}). $ 

Let $a\in \N$. Let $$\text{Per}_a(n)=\{x\in[a]:\sigma^n(x)=x\},$$
be the set of periodic points of order $n$ in $[a]$. The \emph{Gurevich Pressure} of $\phi$ was introduced by Sarig in \cite{sa1}. It is defined by 
\[ P(\phi) = \limsup_{n \rightarrow \infty} \frac{1}{n} \log \sum_{x\in\text{Per}_a(n)}  \exp \left(\sum_{k=0}^{n-1} \phi(\sigma^{k}x)\right).  \]
The value does not depend on the cylinder $[a]$ considered. If $(\Sigma, \sigma)$ is topologically mixing the limit always exists (see \cite{sa1}).

The Gurevich pressure satisfies that
\begin{equation*} \label{eq:press appr}
   P(\phi) = \sup \{ P(\phi|K) : K \subset \Sigma \textrm{ compact and } \sigma\textrm{-invariant} \},
\end{equation*}
where $P(\phi| K)$ is the topological pressure of $\phi$ restricted to the compact set $K$ (for definition and properties see \cite[Chapter 9]{w}). 
Moreover, this notion of pressure satisfies the Variational Principle (see \cite{sa1, ijt}):
\[P(\phi)= \sup \left\{ h_{\mu}(\sigma) + \int \phi \, \d\mu : \mu \in \M_{\sigma} \text{ and } \int  \phi \, \d\mu >- \infty \right\},\]
where  $h_{\mu}(\sigma)$ denotes the entropy of the measure $\mu$ (see  \cite[Chapter 4]{w}). 
If $\phi$ is the zero function we denote $P(0)=h(\sigma)$ and call it the Gurevich entropy of  $\sigma$, see \cite{gu1}. A measure $\mu \in \M_{\sigma}$ attaining the supremum  is called an \emph{equilibrium measure} for $\phi$. It is proved in \cite[Theorem 1.1]{bs} that a potential with summable variations has at most one equilibrium state. Moreover, these are weak Gibbs in the sense of the following Lemma, as established by Sarig in \cite[p.101]{s4} and by Buzzi in \cite[Theorem A.4]{bpp}:

\begin{lemma} \label{weak_gibbs}
Let $(\Sigma, \sigma)$ be a topologically mixing countable Markov shift and $\phi:\Sigma \to \R$ a summable variations function of finite pressure that has an equilibrium measure $\mu$. For every $a,b \in \N$ there exists a constant $C_{ab}>0$ such that for every cylinder of length $n \in \N$ given by $[a, \dots, b]$ and every $x \in [a, \dots, b]$ we have
\begin{equation*}
\frac{1}{C_{ab}}\leq \frac{\mu([a, \dots, b])}{\exp \left( \sum_{i=0}^{n-1}\phi(\sigma^i x) - n P(\phi) \right)}       \leq C_{ab}.
\end{equation*}

\end{lemma}

A countable Markov shift $(\Sigma, \sigma)$ is said to satisfy the Big Images and Pre-images property (BIP property) if there exists a finite set $B \subset \N$ such that for every $a \in \N$ there exists $b,b' \in B$ with $bab'$ an admissible word. In the case when the shift is topologically mixing this is equivalent \cite[p.1752]{sa3} to there existing $N\in\N$ such that for all $i,j\in\N$ and $n\geq N$, $M^n(i,j)>0$, which is referred to as $M$ being finitely primitive in \cite{mubook}.  For example, the full shift on $\N$ satisfies the BIP property.  Under this combinatorial assumption we can sum over all periodic points in the definition of Gurevich pressure (see \cite[Corollary 1]{sa3} and \cite{mubook}). That is, if $(\Sigma, \sigma)$ satisfies the BIP property, then

\begin{equation*}
P(\phi) = \lim_{n \rightarrow \infty} \frac{1}{n} \log \sum_{\sigma^{n}x=x}  \exp \left(\sum_{k=0}^{n-1} \phi(\sigma^{k}x)\right).
\end{equation*}
Moreover, Sarig \cite[Theorem 1]{sa3} established that for a topologically mixing CMS, a function $\phi$ with finite first variation, summable variations with finite pressure has a Gibbs measure if and only if the system has the BIP property. Thus, in this setting,  there exists a constant $C>0$ corresponding to the equilibrium measure $\mu$ such that for every  $x \in [a, \dots, b]$ we have
\begin{equation*}
\frac{1}{C}\leq \frac{\mu([a, \dots, b])}{\exp \left( \sum_{i=0}^{n-1}\phi(\sigma^i x) - n P(\phi) \right)}       \leq C.
\end{equation*}
Hence, contrary to what happens in the general case, the constant $C$ does not depend on the symbols $a$ and $b$.

\subsection{Pressure density}
The next result is a modification of \cite[Proposition 4.1]{itv}, where entropy density is established (see also \cite{ta}), which is more suitable for our purposes.

\begin{lemma}\label{prop:predens}
    Let $(\Sigma,\sigma)$ be a transitive countable Markov shift. Let $\mu\in\M_\sigma$ be an invariant measure with finite entropy. Consider $\epsilon>0$, $\eta>0$, and functions $\{\phi_i\}_{i=1}^n$ in $L^1(\mu)$ with summable variations such that $\sup_{x\in [j]} |\phi_i(x)|<\infty,$ $\forall j\in\N$ and $i\in\{1,\ldots,n\}$. Then, there exists an ergodic measure $\mu_e$ with compact support such that $$h_{\mu_e}(\sigma)>h_\mu(\sigma)-\eta\quad\text{and}\quad\bigg|\int \phi_i \, \d\mu-\int \phi_i \, \d\mu_e\bigg|<\epsilon,$$  for every $i\in\{1,\ldots,k\}$.
\end{lemma}
\begin{proof}
We adapt the proof of \cite[Proposition 4.1]{itv}. We only explain the modifications needed because the functions $\phi_i$ are not assumed to be bounded.

Let $\vartheta$ be the ergodic decomposition of $\mu$. Since $h_\mu(\sigma)<\infty$ and $\phi_i\in L^1(\mu)$, we have
$h_\mu(\sigma)=\int h_\nu(\sigma)\, \d\vartheta(\nu),$ and $\int\phi_i \, \d\mu=\int\left(\int\phi_i \, \d\nu\right)\, \d\vartheta(\nu).$
 Moreover, for $\vartheta$-almost every $\nu$, we have $h_\nu(\sigma)<\infty$ and $\phi_i\in L^1(\nu)$ for every $i$. Approximating the integrable vector-valued function
\begin{align*}
\nu\longmapsto
\left(
h_\nu(\sigma),
\int\phi_1\, \d\nu,\ldots,
\int\phi_n\, \d\nu
\right)
\end{align*}
by simple functions and then rationalizing the coefficients, we obtain ergodic measures $\mu_1,\ldots,\mu_N$, allowing repetitions, such that, for $\bar{\mu}=\frac{1}{N}\sum_{r=1}^N\mu_r,$ we have
\begin{align}
h_{\bar{\mu}}(\sigma)
&>
h_\mu(\sigma)-\frac{\eta}{4},
\label{eq:entropy-approximation}\\
\left|
\int\phi_i \, \d\bar{\mu}
-
\int\phi_i \, \d\mu
\right|
&<
\frac{\epsilon}{3},
\label{eq:integral-approximation}
\end{align}
for $i=1,\ldots,n$. 

Choose $m\in\mathbb{N}$ sufficiently large that $\mu_r(K_m)>\frac{3}{4},$ where $K_m=\bigcup_{j=1}^m[j],$ for every $r=1,\ldots,N$. The proof of \cite[Lemma 4.2]{itv} applies to the functions $\phi_i\in L^1(\mu_r)$, since the relevant step only uses Birkhoff's ergodic theorem and continuity is not required. Consequently, given $\alpha,\beta>0$, for every sufficiently large $p$ there are finite sets
$G_r(p)\subset K_m\cap\sigma^{-p}K_m$ such that
\begin{align}
&\#G_r(p)
\geq
\exp(p(h_{\mu_r}(\sigma)-\alpha)),
\label{eq:cardinality}\\
&\left|
\frac{1}{p}S_p\phi_i(x)-\int\phi_i \, \d\mu_r
\right|
<\beta
\label{eq:birkhoff}
\end{align}
for every $x\in G_r(p)$, every $r$, and every $i$. Moreover, different elements of $G_r(p)$ belong to different cylinders of length $p$. We now repeat the concatenation construction in the proof of \cite[Proposition 4.1]{itv}. By transitivity, there are connecting words of uniformly bounded length, say at most $L$, between symbols in $\{1,\ldots,m\}$. Concatenating the words determined by points in
\begin{align*}
G_1(p),G_2(p),\ldots,G_N(p),
G_1(p),\ldots
\end{align*}
and taking the closure of all these concatenations and their shifts, we obtain a compact invariant set $\Psi_p$. It remains to modify the estimate of the functions $\phi_i$ on $\Psi_p$. Set $V_i=\sum_{\ell=2}^{\infty}\operatorname{var}_{\ell}(\phi_i)<\infty.$ If $x\in G_r(p)$ and $z$ begins with the word $x_0\cdots x_p$, then
\begin{align}
\left|
S_p\phi_i(z)-S_p\phi_i(x)
\right|
\leq
\sum_{j=0}^{p-1}
\operatorname{var}_{p+1-j}(\phi_i)
\leq V_i.
\label{eq:distortion}
\end{align}
Let $\mathcal{C}$ be the finite set of symbols appearing in the chosen connecting words and at their endpoints. By the  boundedness assumption, $B_i=\max_{j\in\mathcal{C}}\sup_{x\in[j]}|\phi_i(x)|<\infty.$ Therefore, the contribution of each connecting word, including the terminal symbol of the principal word is bounded by $(L+1)B_i$. Combining \eqref{eq:birkhoff} and \eqref{eq:distortion}, and following the proof of \cite[Proposition 4.1]{itv}, every ergodic measure $\rho$ supported on $\Psi_p$ satisfies
\begin{align}
\left|
\int\phi_i \, \d\rho
-
\int\phi_i \, \d\bar{\mu}
\right|
\leq
\beta
+
\frac{
V_i+(L+1)B_i
+
(L+1)\left|\int\phi_i \, \d\bar{\mu}\right|
}{p}.
\label{eq:integral-horseshoe}
\end{align}
Choose $\beta<\epsilon/3$ and $p$ sufficiently large that the second term on the right-hand side of \eqref{eq:integral-horseshoe} is smaller than $\epsilon/3$ for every $i$. It follows from \eqref{eq:integral-approximation} that every ergodic measure $\rho$ supported on $\Psi_p$ satisfies
\begin{align*}
\left|\int\phi_i \, \d\rho-\int\phi_i \, \d\mu\right|<\epsilon,
\end{align*}
for $i=1,\ldots,n.$ The entropy estimate is unchanged from the proof of \cite[Proposition 4.1]{itv}. Using \eqref{eq:cardinality}, the independent choices of words from the sets $G_r(p)$ give
\begin{align*}
h_{\mathrm{top}}(\sigma|_{\Psi_p})
\geq
\frac{
p\left(h_{\bar{\mu}}(\sigma)-\alpha\right)
}{
p+1+L
}.
\end{align*}
Choose $\alpha>0$ sufficiently small and then increase $p$ if necessary. By \eqref{eq:entropy-approximation},
\begin{align*}
h_{\mathrm{top}}(\sigma|_{\Psi_p})
>
h_\mu(\sigma)-\eta.
\end{align*}
(If $h_\mu(\sigma)<\eta$, this inequality follows directly from the nonnegativity of topological entropy.) By the variational principle on the compact system $(\Psi_p,\sigma)$, there exists an ergodic measure $\mu_e$ supported on $\Psi_p$ such that $h_{\mu_e}(\sigma)>h_\mu(\sigma)-\eta.$ As explained before we have 
$\left|\int\phi_i \, \d\mu-\int\phi_i \, \d\mu_e\right|<\epsilon,$ for $i=1,\ldots,n.$
Finally, $\operatorname{supp}(\mu_e)\subset\Psi_p,$ and $\Psi_p$ is compact. Hence $\mu_e$ has compact support.\end{proof}

As a consequence of this result we obtain that the set of ergodic measures is pressure dense for potentials with summable variations. 

\begin{theorem} \label{pressure_density}
Let $(\Sigma,\sigma)$ be a transitive countable Markov shift. Let $\phi:\Sigma\to\R$ be a potential of summable variations and finite pressure. Assume that $\sup_{x\in[b]}|\phi(x)|< \infty$ for every $b\in\N$. Let $\mu\in\M_{\sigma}$ be an invariant measure with finite entropy. Then, there exists a sequence $(\mu_n)_n$ of ergodic measures with compact support converging to $\mu$ in the weak$^*$ topology  such that $$\lim_{n\to\infty} \left(h_{\mu_n}(\sigma)+\int \phi \,\d\mu_n \right)=h_\mu(\sigma)+\int \phi \, \d\mu.$$
If $\psi:\Sigma\to\R$ is a function of summable variations and $\sup_{x\in [j]} |\psi(x)|<\infty,$  $\forall j\in\N$, we can further assume that $\lim_{n\to\infty}\int \psi \,\d\mu_n=\int \psi \, \d\mu$.
\end{theorem}

\subsection{Pressure at infinity} We now recall the definition and basic properties of the pressure at infinity as in \cite{v}. Fix $a, q,m\in\N$ and define $$\text{Per}_a(n;q,m)= \bigg\{x \in\ \text{Per}_a(n): \# \big\{k \in \{ 1, \dots , n\} : x_k \leq q\big\} \leq \frac{n}{m} \bigg\}.$$ For a  function $\phi :\Sigma \to \R$ of summable variations define
\begin{equation*}
    Z_n(\phi,a;q,m)=\sum_{x \in \text{Per}_a(n;q,m)} \exp\left( \sum_{i=0}^{n-1}\phi(\sigma^i x) \right).
\end{equation*}
If $\text{Per}_a(n;q,m)= \emptyset$, we set $ Z_n(\phi,a;q,m)=0$. Let
\[ P_{\infty}^{top}(\phi,a;q,m) = \limsup_{n \to \infty} \frac{1}{n}\log Z_n(\phi,a;q,m). \]
The \emph{topological pressure at infinity of $\phi$} is defined by
\[P_{\infty}^{top}(\phi)= \inf_q \inf_m P_{\infty}^{top}(\phi,a;q,m).\]
As observed in \cite[Remark 3.1]{v} this quantity is independent of the choice of $a \in \N$. 

The \emph{measure theoretic pressure at infinity for $\phi$} is defined by
\[ P_{\infty}(\phi)= \sup_{(\mu_n) \to 0} \limsup_{n \to \infty} \left(h_{\mu_n}(\sigma) + \int \phi \, \d \mu_n \right),\]
where the supremum is taken over all sequences $(\mu_n)$ in $\M_{\sigma}$ such that for every $n \in \N$ we have $\int \phi \, \d \mu_n>-\infty$, which converges to zero in the cylinder topology. It was established in \cite[Theorem 1.2]{v} that if $\phi$ is a function of summable variations and $(\Sigma,\sigma)$ a transitive CMS then $P_{\infty}(\phi)=P_{\infty}^{top}(\phi)$. We will, therefore, refer to this quantity as the \emph{pressure at infinity for $\phi$}. If $\phi$ is the zero function then we refer to $P_{\infty}(\phi)$ as the \emph{entropy at infinity} and we will denote it by $h_{\infty}(\sigma)$. This quantity was defined and studied  in \cite{bu, itv}.

In \cite{sa2}, Sarig introduced the class of strongly positive recurrent potentials (referred to as SPR). It turns out that the thermodynamic formalism of SPR functions is very well behaved. Indeed, it  is similar to that of H\"older potentials in subshifts of finite type.  In \cite[Theorem 8.2]{rs} it is shown that a weakly H\"older function $\phi$ is SPR if and only if $P_{\infty}(\phi)<P(\phi)$. The same holds for functions of summable variations with the caveat that the definition of SPR requires regularity of the induced function and, for summable variation functions,  that needs to be assumed. In this case, the function $\phi$ admits a Ruelle–Perron–Frobenius measure $\mu_\phi$, which is a probability measure with strong ergodic properties (see \cite{sa2}). Moreover, if $\int \phi \, \d\mu_\phi>-\infty$, then $\mu_\phi$ is the unique equilibrium state of $\phi$.

\subsection{Derivatives of the pressure and the pressure at infinity}

Let $Q: \R \to \R$ be a convex function. A \emph{support line} for $Q$ at $t=0$ is an affine function of the form $Q(0) +Dt$ with the property that for every $t \in \R$ we have
\begin{equation*}
Q(t) \geq Q(0) +Dt.
\end{equation*}
The \emph{subgradient} of $Q$ at $t=0$, that we denote by $\partial Q(0)$, is the set of all slopes $D \in \R$ of support lines of $Q$ at $t=0$.  A convex function $Q: \R \to \R$ has right and left derivatives. In particular, $$\partial Q(0)= [ Q'_-(0), Q'_+(0)],$$ where $ Q'_-(0), Q'_+(0)$ denote the left and right derivatives of $Q(t)$ at $t=0$, respectively. Hence $Q$ is differentiable at $t=0$ if and only if the cardinality of $\partial Q(0)$ is one  (see \cite[Theorem 25.1]{ro}). Moreover, a convex function on $\R$ is non-differentiable at at most countably many points.

\begin{lemma}
\label{lem:local-convex}
Let $J\subset \R$ be an open interval such that $0 \in J$, and let $q_n,q:J\to \R$ be convex functions. Assume that each $q_n$ is differentiable and that $q_n\to q$ pointwise on $J$.
Then, for every $D\in\partial q(0)$, there exist a subsequence, still denoted by $(q_n)_n$, and a sequence $t_n\to0$ such that $q_n'(t_n)\rightarrow D.$ Moreover, $q_n(t_n)-t_nq_n'(t_n)\rightarrow q(0)$.
\end{lemma}

\begin{proof}
This is a minor refinement  of the argument in \cite[Proposition~2.5(i)]{fe}. Choose $\delta_k\downarrow0$ with $[-\delta_k,\delta_k]\subset J$, and set
\[
    a_{n,k}:=
    \frac{q_n(0)-q_n(-\delta_k)}{\delta_k},
    \qquad
    b_{n,k}:=
    \frac{q_n(\delta_k)-q_n(0)}{\delta_k}.
\]
For fixed \(k\), these numbers converge respectively to
\[
    a_k:=
    \frac{q(0)-q(-\delta_k)}{\delta_k},
    \qquad
    b_k:=
    \frac{q(\delta_k)-q(0)}{\delta_k}.
\]
Since  $D\in\partial q(0)$, we have $a_k\leq D\leq b_k$. Choose $n_k$ large enough so that the distance from $D$ to $[a_{n_k,k},b_{n_k,k}]$ is smaller than $1/k$, and let $d_k$
be the point of this interval closest to $D$. By the mean value theorem and the intermediate value property of derivatives, there exists $t_k\in[-\delta_k,\delta_k]$ such that
$q_{n_k}'(t_k)=d_k.$ Thus $t_k\to0$ and $q_{n_k}'(t_k)\to D$. Since pointwise convergence of finite convex functions on an open interval is locally uniform, $q_{n_k}(t_k)\to q(0)$, which proves the last assertion after relabeling the subsequence.
\end{proof}

\begin{proposition}
\label{der}
Let $(\Sigma,\sigma)$ be a transitive countable Markov shift, and let $\phi,\psi:\Sigma\to\R$ be functions of summable
variations such that, for every  $j\in\N$, we have $\sup_{x\in[j]}|\phi(x)|<\infty$ and  $\sup_{x\in[j]}|\psi(x)|<\infty$. Set $Q(t):=P(\phi+t\psi)$ and  $Q_\infty(t):=P_\infty(\phi+t\psi)$.
\begin{enumerate}
\item[(a)]
Assume that $Q$ is finite on an open interval containing $0$. Then, for $D\in\R$, there exists a sequence $(\mu_n)_n\subset\mathcal M_\sigma$ such that
\[
    h_{\mu_n}(\sigma)+\int\phi\, \d\mu_n  \rightarrow P(\phi)  \quad\text{and}\quad \int\psi\, \d\mu_n \rightarrow D
\]
if and only if $D\in\partial Q(0)$. Moreover, if $\mu$ is an equilibrium state for $\phi$ and $\psi\in L^1(\mu)$, then $\int\psi\,\d\mu\in\partial Q(0)$.

\item[(b)]
Assume that $Q_\infty$ is finite on an open interval containing $0$. Then, for $D\in\R$, there exists a sequence $(\mu_n)_n\subset\mathcal M_\sigma$ such that
$\mu_n\rightarrow0$ in the cylinder topology,
\[
    h_{\mu_n}(\sigma)+\int\phi\,\d\mu_n \rightarrow P_\infty(\phi), \qquad \int\psi\,\d\mu_n\rightarrow D,
\]
if and only if $D\in\partial Q_\infty(0)$.
\end{enumerate}
In both parts, the measures $\mu_n$ may be chosen ergodic and compactly supported.
\end{proposition}

\begin{proof}
Set $\mathcal F_\phi(\mu):= h_\mu(\sigma)+\int\phi\,\d\mu$. If $\mathcal F_\phi(\mu_n)\to P(\phi)$ and $\int\psi\,\d\mu_n\to D$, then, for every $t$ for which $Q(t)$
is finite, we have $Q(t)    \geq \mathcal F_\phi(\mu_n)+t\int\psi\,\d\mu_n.$  Passing to the limit gives $Q(t)\geq Q(0)+tD$, and hence $D\in\partial Q(0)$. The same argument, using the definition of pressure at infinity, proves the corresponding implication in part~(b). If $\mu$ is an equilibrium state for $\phi$ and $\psi\in L^1(\mu)$, then
\[
    Q(t)\geq h_\mu(\sigma)+\int(\phi+t\psi)\,\d\mu = Q(0)+t\int\psi\,\d\mu,
\]
so $\int\psi\,\d\mu\in\partial Q(0)$.

We prove the converse in part~(a). By compact approximation of the Gurevich pressure, there exists a sequence of compact transitive subshifts of finite type $(\Sigma_n,\sigma)$ such that $Q_n(t):=P_{\Sigma_n}(\phi+t\psi) \rightarrow Q(t)$ pointwise on an open interval containing $0$. Each $Q_n$ is differentiable. If $\mu_{n,t}$ is the equilibrium state of
$\phi+t\psi$ on $\Sigma_n$, then $Q_n'(t)=\int\psi\,\d\mu_{n,t}.$ Let $D\in\partial Q(0)$. By Lemma~\ref{lem:local-convex}, after passing to a subsequence there
exist $t_n\to0$ such that $Q_n'(t_n)\to D$. Set $\mu_n:=\mu_{n,t_n}$. Then each $\mu_n$ is ergodic and compactly
supported, and $\int\psi\,\d\mu_n=Q_n'(t_n)\rightarrow D$. Moreover,  $\mathcal F_\phi(\mu_n)= Q_n(t_n)-t_nQ_n'(t_n)  \rightarrow Q(0)=P(\phi).$
This proves part~(a).

We now prove the converse in part~(b). We first observe that, in the definition of $P_\infty(\phi+t\psi)$, one may restrict to sequences
of ergodic, compactly supported measures. Indeed, apply Theorem ~\ref{pressure_density} to the $n$-th term of an escaping sequence, controlling $\phi+t\psi$ and the first $n$ cylinders. The resulting sequence still converges to zero and has no smaller limsup of the free energies. Set $D_-:=(Q_\infty)'_-(0)$ and $D_+:=(Q_\infty)'_+(0)$.
We first construct a sequence corresponding to $D_+$. Choose $\varepsilon_n\downarrow0$, put $r_n:=\varepsilon_n/2$, and choose $\eta_n>0$ such that $\eta_n/r_n\to0$. For each $n$, choose an escaping sequence $(\rho_{n,k})_k$ of ergodic, compactly supported measures and pass to a subsequence in $k$ so that
\[
    \mathcal F_\phi(\rho_{n,k})     +\varepsilon_n\int\psi\,d\rho_{n,k} \rightarrow L_n,
\]
where $Q_\infty(\varepsilon_n)-\eta_n     \leq L_n\leq Q_\infty(\varepsilon_n)$. Comparing the same sequence at the parameters $\varepsilon_n-r_n$ and $\varepsilon_n+r_n$, and passing to a
further subsequence, we may assume that $\int\psi\,\d\rho_{n,k}\to d_n^+$, where
\[
 \frac{Q_\infty(\varepsilon_n) -Q_\infty(\varepsilon_n-r_n)-\eta_n}{r_n} \leq d_n^+ \leq
 \frac{Q_\infty(\varepsilon_n+r_n)  -Q_\infty(\varepsilon_n)+\eta_n}{r_n}.
\]
Convexity and $\eta_n/r_n\to0$ imply that $d_n^+\to D_+$. A diagonal choice therefore gives ergodic, compactly supported
measures $\mu_n^+\to0$ such that
\[
    \int\psi\,\d\mu_n^+\longrightarrow D_+,\qquad \mathcal F_\phi(\mu_n^+)\longrightarrow Q_\infty(0).
\]
Applying the same argument at $-\varepsilon_n$ gives a sequence $\mu_n^-\to0$ satisfying
\[
    \int\psi\, \d\mu_n^-\longrightarrow D_-,\qquad  \mathcal F_\phi(\mu_n^-)\rightarrow Q_\infty(0).
\]
Let $D\in\partial Q_\infty(0)=[D_-,D_+]$, and choose $\theta\in[0,1]$ such that $D=\theta D_-+(1-\theta)D_+.$ Set
$\lambda_n:=\theta\mu_n^-+(1-\theta)\mu_n^+.$ Since entropy is affine on invariant measures,
\[
    \lambda_n\to0, \qquad   \mathcal F_\phi(\lambda_n)\to Q_\infty(0), \qquad \int\psi\,\d\lambda_n\to D.
\]
It remains only to recover ergodicity. Enumerate the cylinders of $\Sigma$ as $(C_j)_{j\geq1}$. Apply Theorem ~\ref{pressure_density} to $\lambda_n$, controlling
$\phi,\psi,    \chi_{C_1},\ldots,\chi_{C_n},$ with entropy loss and integral errors smaller than $1/n$. We obtain an ergodic, compactly supported measure $\widehat\mu_n$
such that $\widehat\mu_n\to0 , \int\psi\,\d\widehat\mu_n\to D$, and
\[
    \liminf_{n\to\infty}  \mathcal F_\phi(\widehat\mu_n) \geq Q_\infty(0).
\]
The reverse limsup inequality follows from the definition of $P_\infty(\phi)$, since $\widehat\mu_n\to0$. Hence
\[
    \mathcal F_\phi(\widehat\mu_n) \rightarrow Q_\infty(0)=P_\infty(\phi).
\]
This proves part~(b).
\end{proof}

\begin{example}[Random walk] \label{rw}

Consider the CMS defined by the graph in $\N$ such that $n \mapsto n$ and for $n >1$ we also have that $n \mapsto n+1$  for $n \geq 1$ and $n \mapsto n-1$ for $n >1$. This is a topologically mixing CMS of finite entropy that does not have a measure of maximal entropy. Let $a, b \in \R$ be two negative numbers. Consider a function $\phi:\Sigma \to \R$ be defined by
\begin{equation*}
\phi(x)=
\begin{cases}
a & \text{ if  } x_1 \text{ is odd};\\
b & \text{ if  } x_1 \text{ is even}
    \end{cases}
\end{equation*}
This function, being bounded, has finite pressure. Moreover, it does not have an equilibrium measure. Indeed, if $\mu$ was an equilibrium measure for $\phi$ the normalization of the restriction of $\mu$ to CMS obtained by deleting the vertices $\{1, 2\}$ is also an equilibrium measure (since these two countable Markov shifts are conjugated). But there is at most one equilibrium measure for functions of summable variations \cite[Theorem 1.1]{bs}. Therefore, $\phi$ is not strongly positive recurrent and by \cite{rs, v} we have $P(\phi)=P_{\infty}(\phi)$. The same argument holds for $t\phi$ with $t \in \R$. Hence, for every $t \in \R$ we have $P(t\phi)=P_{\infty}(t \phi)$. Therefore, the subgradients of the pressure and that of the pressure at infinity coincide.

\end{example}

\section{Suspension flows over countable Markov shifts} \label{sec:susp}
In this section we recall basic definitions and properties of suspension flows defined over countable Markov shifts. We  discuss topologies in the space of invariant measures that make the space of invariant sub-probabilities compact. We also prove that  ergodic invariant probability measures of compact support are pressure dense.

Let $(\Sigma, \sigma)$ be a transitive countable Markov shift and $\tau \colon \Sigma \to \R^+$ be a positive continuous function such that for every  $x\in\Sigma$ we have
\begin{equation} \label{eq:Hopf_cond}
\sum_{i=0}^{\infty}\tau(\sigma^i x)=\infty.
\end{equation}
Let
\begin{equation*}\label{eq:flow phase }
Y=\{ (x,t)\in \Sigma  \times \R \colon 0 \le t \le\tau(x)\}/\sim
\end{equation*}
where $(x,\tau(x))\sim (\sigma(x),0)$ for
each $x\in \Sigma $. The \emph{suspension semi-flow} over $\sigma$ with \emph{roof function} $\tau$ is the semi-flow $\Phi = (
\Phi_t)_{t \ge 0}$ on $Y$ defined by
\[
 \Phi_t(x,s)= (x,
s+t) \ \text{whenever $s+t\in[0,\tau(x)]$.}
\]
Note that condition (\ref{eq:Hopf_cond}) guarantees that the flow is complete; that is, $\Phi_t:Y\to Y$ is well defined for all $t \ge 0$.

\subsection{Invariant measures}  \label{ssec:meas AK}
There is a close relation between invariant measures for the flow and invariant measures for the base map.
A probability measure $\nu$  on $Y$  is
\emph{$\Phi$-invariant} if $\nu(\Phi_t^{-1}A)= \mu(A)$ for every
$t \ge 0$ and every measurable set $A \subset Y$. Denote by $\M_\Phi$ the space of $\Phi$-invariant probability
measures on $Y$. Let,
\begin{equation}
\M_\sigma(\tau):= \left\{ \mu \in \mathcal{M}_{\sigma}: \int \tau \, \d \mu < \infty \right\}.
\end{equation}
 Denote by $\text{Leb}$  the one-dimensional Lebesgue measure and $\mu \in \M_\sigma(\tau)$, then  
\begin{equation} \label{susp_meas}
\frac{(\mu \times \text{Leb})|_{Y} }{(\mu \times \text{Leb})(Y)} \in \M_{\Phi}.
\end{equation}

If $(\Sigma , \sigma)$ is a countable Markov shift and $\tau:\Sigma \to \R$ is uniformly bounded away from zero, then equation \eqref{susp_meas} induces a bijection between $\mathcal{M}_{\sigma}(\tau)$ and $\M_{\Phi}$. Note, however, that if  $\tau:\Sigma \to \R$ is  not  bounded away from zero then it is possible for an infinite $\sigma$-invariant measure $\mu$ to have $\int  \tau \, \d\mu <\infty$. In this case the measure $(\mu \times \text{Leb})|_{Y} /(\mu \times \text{Leb})(Y) \in \M_\Phi$.

Given a continuous function $g \colon Y\to\R$ we define the function
$\Delta_g\colon\Sigma\to\R$~by
\[
\Delta_g(x)=\int_{0}^{\tau(x)} g(x,t) \,   \d t.
\]
The function $\Delta_g$ is also continuous, moreover if $\nu=\frac{(\mu \times \text{Leb})|_{Y} }{(\mu \times \text{Leb})(Y)} \in \M_{\Phi},$ then
\begin{equation*} \label{eq:rela}
\int_{Y} g \, \d \nu= \frac{\int_\Sigma \Delta_g\,\d
\mu}{\int_\Sigma\tau \, \d \mu}.
\end{equation*}

\begin{remark}\label{rem:ex} Given a continuous function $\phi:\Sigma\to\R$, there exists a continuous function $g:Y\to\R$ such that $\Delta_g=\phi$ (see \cite{brw}). For instance, consider $$g(x,t)=\frac{\phi(x)}{\tau(x)}\psi'\bigg(\frac{t}{\tau(x)}\bigg),$$ for every $x\in \Sigma$ and $t\in [0,\tau(x)]$, where $\psi:[0,1]\to[0,1]$ is any nondecreasing $C^1$ function such that $\psi(0)=0$, $\psi(1)=1$, and $\psi'(0)=\psi'(1)=0$. We can assume, in particular, that 
$$\frac{1}{2}\frac{\phi(x)}{\tau(x)}\le |g(x,t)|\le 2\frac{\phi(x)}{\tau(x)}.$$
\end{remark}

The entropy of a flow with respect to an invariant measure $\nu$,  denoted $h_{\nu}(\Phi)$, is defined as  the entropy of the corresponding time one map. The following  classical result obtained by Abramov \cite{a} relates the entropy of a measure for the flow with the entropy of the corresponding measure for the base map: let $\nu \in \M_{\Phi}$ be such that  $\nu=(\mu \times \text{Leb})|_{Y} /(\mu \times \text{Leb})(Y)$, where $\mu \in \M_{\sigma}$, then
\begin{equation}
h_{\nu}(\Phi)=\frac{h_{\mu}(\sigma)}{\int \tau \, \d \mu}.
\end{equation}
\label{prop:Abr}

\subsection{Topologies in the space of flow invariant measures} \label{topo-flow}
As in section \ref{topologies} it is possible to define the weak* topology in $\M_{\Phi}$. Indeed, denote by $C_b(Y)$ the space of bounded and continuous functions on $Y$. A sequence of measures $(\nu_n)_n$ in $\M_{\Phi}$ converges in the weak* topology to $\nu \in \M_{\Phi}$ if for every $f \in C_b(Y)$  we have $\lim_{n \to \infty} \int f \, \d \nu_n = \int f \, \d \nu$. This notion of convergence induces a metrizable topology on $\M_{\Phi}$. 
Again, if $\Sigma$ is compact then so is $\M_{\Phi}$.

In \cite[Section 6]{iv} the notion of cylinder topology was extended to suspension flows.  Let  
 $\M_{\le 1}(\Phi)$ be the   space of invariant sub-probability measures on $Y$. Let $(\nu_n)_n$ and $\nu$ be measures in $\M_{\le 1}(\Phi)$.
We say that $(\nu_n)_n$ \emph{converges on cylinders} to $\nu$ if $$\lim_{n\to\infty}\nu_n(C\times [a,b])=\nu(C\times[a,b]),$$
for every cylinder $C\subset \Sigma$ and $a,b\in\R$. This notion of convergence induces a metrizable topology on $\M_{\le 1}(\Phi)$. We call this topology the \emph{cylinder topology}. It coincides with the weak* topology when there is no escape of mass (see \cite[Lemma 6.7]{iv}).

\begin{remark}\label{rem:topsus}
    Let $\nu$ and $(\nu_n)_n$ be invariant probability measures for the suspension flow with $\nu_n$  the normalization of $(\mu_n \times \text{Leb})$ and $\nu$ that of $(\mu \times \text{Leb})$. The following statements are equivalent (see \cite[Lemma 8.1.1]{iv}):
\begin{enumerate}
\item The sequence $(\nu_n)_n$ converges on cylinders to $\lambda \nu$, where $\lambda\in [0,1]$. 
\item The following limit holds $$\lim_{n\to\infty}\frac{\mu_n(C)}{\int \tau \d\mu_n}=\lambda\frac{\mu(C)}{\int \tau \d\mu},$$
for every cylinder $C\subset \Sigma$.
\end{enumerate}
Assume $\tau$ is bounded away from zero. Note that if $\lim_{n\to\infty}\int \tau \d\mu_n=\infty$, or $(\mu_n)_n$ converges to the zero measure, then the sequence $(\nu_n)$ converges to the zero measure. Otherwise, we have that $(\mu_n)_n$ converges on cylinders to $\lambda_1\mu$ and $\lim_{n\to\infty}\int \tau \d\mu_n=\lambda_2\int \tau \d\mu$, where $\lambda=\lambda_1/\lambda_2$, up to subsequences.
\end{remark}

\subsection{Pressure for suspension flows}
In this section we recall the thermodynamic formalism of suspension flows over countable Markov shifts as developed in \cite{bi1,ijt, jkl,ke,sav} among others.

Let $(\Sigma, \sigma)$ be a  transitive CMS, and let $\tau: \Sigma \to (0,\infty)$ be a function of summable variations which is bounded away from zero. Denote by  $(Y, \Phi)$ the corresponding suspension flow.

Let $a\in\N$ be a symbol in the alphabet.  For $T_1\le T_2$ we define 
$$\mathrm{FPer}_a(T_1,T_2)=\{(x;s)\in[a]\times [T_1,T_2]:\Phi_s(x,0)=(x,0)\}.$$
The set $\text{FPer}_a(T_1,T_2)$ corresponds to the collection of periodic orbits starting in $[a] \times \{0\}$ whose periods lie in the interval $[T_1,T_2]$.

Let $g:Y \to \R$ be a continuous function such that $\Delta_g:\Sigma \to \R$ is of summable variations. The \emph{pressure of $g$} is defined by
\begin{equation*}
    P^{\Phi}(g)=\limsup_{T \to \infty} \frac{1}{T} \log \left(\sum_{(x;s)\in \mathrm{FPer}_a(T-d_1,T+d_2)} \exp\left(\int_0^{s} g(\Phi_t(x,0)) \, \d t \right) \right). 
\end{equation*}

This definition is independent of $a\in\N$, $d_1>0$ and $d_2\ge 0$. Moreover, if $(\Sigma,\sigma)$ is topologically mixing, then the limsup is in fact a limit. This notion of pressure satisfies the following properties:
\begin{eqnarray*}
P^{\Phi}(g)
&=& \inf\{t \in \R : P (\Delta_g - t \tau) \leq 0\} =\sup \{t \in \R : P(\Delta_g - t \tau) \geq 0\} \\
&=& \sup \left\{ h_{\nu}(\Phi) +\int g \, \d \nu : \nu\in
\mathcal{M}_{\Phi} \text{ and } \int_Y g \,  \d\nu >-\infty \right\}\\
&=& \sup \{ P(g|K) : K\in \cK \},
\end{eqnarray*}
where $\cK$ is the set of all compact  $\Phi$-invariant sets (see \cite{jkl}). A measure $\nu \in \M_{\Phi}$ such that 
$P^{\Phi}(g)= h_{\nu}(\Phi) +\int g \, \d \nu$ is an \emph{equilibrium measure for $g$}.  If $g=0$ is the zero function then we denote $h(\Phi):=P^{\Phi}(0)$ and call it the \emph{entropy} of $\Phi$. The equilibrium measure for the zero potential is called a \emph{measure of maximal entropy}.

\subsection{Pressure density for suspension flows}
We now establish a version of Theorem \ref{pressure_density} for suspension flows over countable Markov shifts.

\begin{theorem} \label{pressure_density_susp}
Let $(\Sigma,\sigma)$ be a transitive countable Markov shift, and $\tau:\Sigma\to\R$ a roof function bounded away from zero with summable variations and such that $\sup_{x\in[b]}\tau(x)< \infty$ for every $b\in\N$. Let $(Y,\Phi)$ be the associated suspension flow. Let $\nu\in\M_\Phi$ be an invariant measure with finite entropy. Let $g:Y\to\R$ be a bounded continuous function such that $\Delta_g$ has summable variations and $\sup_{x\in[b]}|\Delta_g(x)|< \infty$ for every $b\in\N$. Then, there exists a sequence $(\nu_n)_n$ of ergodic measures with compact support converging to $\nu$ in the weak$^*$ topology  such that $$\lim_{n\to\infty} \left(h_{\nu_n}(\Phi)+\int g \, \d\nu_n \right)=h_\nu(\Phi)+\int g \, \d\nu.$$
\end{theorem}
\begin{proof}
Let $\mu$ be the measure on $\Sigma$ associated to $\nu$. Applying Theorem \ref{pressure_density} to the functions $\tau$ and $\Delta_g$ we obtain a sequence of ergodic $\sigma$-invariant measures $(\mu_n)_n$ that converges to $\mu$ and such that $h_{\mu_n}(\sigma)\to h_\mu(\sigma)$, $\int \tau \, \d\mu_n\to\int \tau \, \d\mu$, $\int \Delta_g \, \d\mu_n\to\int \Delta_g \, \d\mu$. Set $\nu_n$ as the invariant measure on $Y$ associated to $\mu_n$. The desired limit follows by Kac's and Abramov's formula.\end{proof}

\section{Pressure at infinity for suspension flows} \label{pres_inf_flow}
Let $(\Sigma,\sigma)$ be a  transitive countable Markov shift, and let $\tau:\Sigma\to (0,\infty)$ be a function of summable variations bounded away from zero. Consider the associated suspension flow $(Y,\Phi)$. In this section, we define and establish a formula for the topological pressure at infinity of potentials over $(Y,\Phi)$, and relate it to the measure-theoretic pressure at infinity. From both topological and measure-theoretic perspectives, these quantities measure the degree of disorder of the system in neighborhoods of infinity. These notions extend the concept of entropy at infinity for suspension semi-flows over countable Markov shifts, previously introduced in \cite{irv1}.

Define $$Y_0= \{ (x,t)\in \Sigma  \times \R \colon 0 \le t <\tau(x)\}.$$ 
Let $a\in\N$ be a symbol in the alphabet. 
Fix $\tau_0 \in (0,\inf \tau)$ and  $Q,M\in\N$. Define  $K_Q= \bigcup_{i=1}^Q [i]$ , $\widetilde{K}_Q=(K_Q\times\R)\cap Y_0$ and $\tau_Q=\sup\{\tau(x):x_1\leq Q\}$. Note that to a periodic point $y=(x,0)\in Y$ we associate a periodic measure $\nu_x\in\M(\Phi)$. Define
$$\mathrm{FPer}_a(T_1,T_2;Q,M)=\bigg\{(x;s)\in \mathrm{FPer}_a(T_1,T_2): \nu_x(\widetilde{K}_Q)\le \frac{1}{M}\bigg\}.$$
Note that if $(x;s)\in\mathrm{FPer}_a(T_1,T_2;Q,M)$, then $(x;s)\in \mathrm{FPer}_a(T_1,T_2)$ and \begin{align}\label{ineq:QM}
    S_n\chi_{K_Q}(x) \le \frac{1}{\tau_0 M}S_n\tau(x),
\end{align}
where $n\in\N$ satisfies that $S_n\tau(x)=s$.

\begin{definition} Let $g:Y\to\R$ be a continuous function such that $\Delta_g$ has summable variations. Let $Q,M\in \N$. Define $$Z_{T_1,T_2}(g,a;Q,M)=\sum_{(x;s)\in  \mathrm{FPer}_a(T_1,T_2;Q,M)}\exp\left(\int_0^{s} g(\Phi_t(x,0)) \, \d t\right).$$
If $\mathrm{FPer}_a(T_1,T_2;Q,M)$ is empty, we set $Z_{T_1,T_2}(g,a;Q,M)=0$. Let $d>0$. Define
\begin{align*}P_\infty^\Phi(g,a; Q,M)&=\limsup_{T\to\infty}\frac{1}{T}\log Z_{T-d,T}(g,a;Q,M),\quad\text{and}\quad P_\infty^\Phi(g,a)=\inf_M\inf_Q P_\infty^\Phi(g,a; Q,M).
\end{align*}
\end{definition}

Since the value $P_\infty^\Phi(g,a)$ is independent of $a$ and $d$ (see Lemma \ref{lema_var} below), we refer to it as the \emph{pressure at infinity} of $g$, denoted by $P_\infty^\Phi(g)$. If $g \equiv 0$, we call this quantity the \emph{entropy at infinity} and denote it by $h_{\infty}(\Phi)$.

\subsection{Variational principle at infinity}
We now prove a variational principle for the pressure at infinity. The relevant measures are sequences converging to the zero measure in the cylinder topology. Indeed, we will prove that under suitable assumptions:
$$P_\infty^\Phi(g)=\sup_{(\nu_n)_n\to 0}\limsup_{n\to\infty} \left( h_{\nu_n}(\Phi)+\int g \,\d\nu_n \right).$$ 
The following result establishes a formula for the pressure at infinity which in particular proves that the pressure at infinity is independent of the symbol $a\in\N$. This is a continuous version of \cite[Proposition 3.2]{v}.

\begin{lemma} \label{lema_var} Let $(\Sigma,\sigma)$ be a transitive countable Markov shift, and $\tau:\Sigma\to\R$ a roof function bounded away from zero and summable variations. 
Let $(Y,\Phi)$ be the associated suspension flow  and  $g:Y\to\R$ be a continuous function such that $\Delta_g$ has summable variations. Let $V:Y\to\R$ be a continuous function such that $\Delta_V=\sum_{i\in\N} \frac{1}{i} \chi_{[i]}$. Then, 
$$P_\infty^\Phi(g,a)=\lim_{t\to \infty}P^\Phi(g-tV),$$
for every $a\in\N$.
    \end{lemma}
    
    \begin{proof}
Notice that $P_\infty^\Phi(g)\ge P_\infty^\Phi(g-tV)$, for every $t\ge 0$. If $(x;s)\in \text{FPer}_a(T-d,T;Q,M)$, then 
   $$\int_0^s V(\Phi_t(x,0))\,\d t=S_n(\Delta_V\chi_{K_Q})(x)+S_n(\Delta_V\chi_{K^c_Q})(x)\le \frac{s}{\tau_0 M}+\frac{s}{Q}\le T\bigg(\frac{1}{\tau_0 M}+\frac{1}{Q}\bigg),$$
where $n\in\N$ satisfies  $S_n\tau(x)=s$, and we used (\ref{ineq:QM}) together with the definition of $\Delta_V$. It follows that 
$$Z_{T-d,T}(g-tV,a;Q,M)\ge Z_{T-d,T}(g,a;Q,M)\exp\bigg(-T\bigg(\frac{1}{\tau_0 M}+\frac{1}{Q}\bigg)\bigg),$$
and thus 
$$P_\infty^\Phi(g-tV, a; Q,M)\ge P_\infty^\Phi(g, a; Q,M)-T\bigg(\frac{1}{\tau_0 M}+\frac{1}{Q}\bigg).$$
We conclude that $P_\infty^\Phi(g-tV, a)\ge P_\infty^\Phi(g, a)$, and therefore $P_\infty^\Phi(g-tV, a)= P_\infty^\Phi(g, a)$. Since for all $t$ we have $P_\infty^\Phi(g-tV, a)\le P^\Phi(g-tV, a)$, we obtain \begin{align}\label{ineq:33}P_\infty^\Phi(g)\le \lim_{t\to\infty}P^\Phi(g-tV).
\end{align}
We will now show that the    reverse inequality holds. 

Set $\text{FPer}^c_a(T;Q,M)=\text{FPer}_a(T-d,T)\setminus \text{FPer}_a(T-d,T;Q,M)$. Note that if $(x;s)\in \text{FPer}^c_a(T;Q,M)$, then
$ S_n\chi_{K_Q}(x) > \frac{s}{\tau_Q M}$ (see (\ref{ineq:QM})), and therefore
  $$\int_0^s V(\Phi_t(x,0))\,\d t\ge S_n(\Delta_V\chi_{K_Q})(x)\ge \frac{s}{\tau_Q MQ}\ge \frac{T-d}{ \tau_Q MQ}.$$
Then, 
\begin{align*}
\sum_{(x,s)\in  \mathrm{FPer}^c_a(T;Q,M)}\exp\left(\int_0^{s} (g-tV)(\Phi_t(x,0)) \, \d t\right) &\le  e^{-\frac{(T-d)t}{ \tau_Q MQ}}\sum_{(x,s)\in  \mathrm{FPer}^c_a(T;Q,M)}\exp\left(\int_0^{s} g(\Phi_t(x,0)) \, \d t\right)
\end{align*}
By the definition of the pressure, there exists $t_1>0$ such that if $T\ge t_1$, then 
\begin{align*}
\sum_{(x;s)\in  \mathrm{FPer}^c_a(T;Q,M)}\exp\left(\int_0^{s} g(\Phi_t(x,0)) \, \d t\right)\le \exp\bigg(T(P^\Phi(g)+1)\bigg).
\end{align*}

Note that if $T\ge t_1$ and $t>\frac{T}{T-d}\tau_Q MQ(P^\Phi(\phi)+1+\kappa),$ we have that 
$$\sum_{(x;s)\in  \mathrm{FPer}_a(T-d,T)}\exp\left(\int_0^{s} (g-tV)(\Phi_t(x,0)) \, \d t\right)\le Z_{T-d,T}(g-tV,a;Q,M) +e^{-\kappa T},$$
for every $\kappa>0$. We conclude that $P^\Phi(g-tV)\le \max\{P^\Phi_\infty(g,a;Q,M),-\kappa\}$, and therefore $$\lim_{t\to\infty}P^\Phi(g-tV)\le P_\infty^\Phi(g,a;Q,M).$$
This bound, together with inequality (\ref{ineq:33}) finishes the proof. 
  \end{proof}

\begin{theorem} \label{prin_var_flow_infty}
Let $(\Sigma,\sigma)$ be a transitive countable Markov shift, and $\tau:\Sigma\to\R$ a roof function bounded away from zero with summable variations. Moreover, assume that $\sup_{x\in[b]}\tau(x)< \infty$ for all $b\in\N$. Let $(Y,\Phi)$ be the associated suspension flow
and  $g:Y\to\R$ be a continuous function such that $\Delta_g$ is of summable variations. Then, $$P_\infty^\Phi(g)=\sup_{(\nu_n)_n\to 0}\limsup_{n\to\infty} \left( h_{\nu_n}(\Phi)+\int g \,\d\nu_n \right),$$ where the supremum runs over sequences $(\nu_n)_n$ converging to the zero measure. 
\end{theorem}

\begin{proof}
Let $(\nu_n)_n$ be a sequence in $\M_{\Phi}$ that converges to zero on the cylinder topology  and
$V:Y\to\R$ a function such that $\Delta_V=\sum_{i\in\N} \frac{1}{i} \chi_{[i]}$. Then,  
$$ \lim_{n \to \infty} \int V \, \d \nu_n =0.$$
By the variational principle for the pressure on the flow we have that, for every $t \in \R$

\begin{align*}
\limsup_{n \to \infty} \left( h_{\nu_n}(\Phi)  + \int g \, \d \nu_n \right) =
\limsup_{n \to \infty} \left(h_{\nu_n}(\Phi)  + \int (g - t V) \, \d \nu_n \right) \leq P^{\Phi}(g-tV).
\end{align*}
by Lemma \ref{lema_var}, we obtain
\begin{align*}
\limsup_{n \to \infty} \left( h_{\nu_n}(\Phi)  + \int g \, \d \nu_n \right) \leq \lim_{t \to \infty}
P^{\Phi}(g-tV)=P_\infty^\Phi(g).
\end{align*}
In order to prove the other inequality, assume first that $P_\infty^\Phi(g)>-\infty$. From Lemma \ref{lema_var}, for every $t \in \R$ we have $P^{\Phi}(g -tV) \geq P_\infty^\Phi(g)$. The variational principle implies the existence of a sequence $(\nu_n)_n$ in $\M_{\Phi}$, satisfying for every $n \in \N$:
\begin{align} \label{e_var_1}
h_{\nu_n}(\Phi)+ \int (g-nV) \, \d \nu_n \geq P^{\Phi}(g -n V) - \frac{1}{n},
\end{align}
with $ \int (g-nV) \, \d \nu_n > -\infty$. Thus,  $ \int g \, \d \nu_n > -\infty$. The inequality will follow once we prove that $(\nu_n)_n$ converges to zero. Assume by way of contradiction that there exists a cylinder $C \subset \Sigma$ such that $\limsup_{n\to\infty} \nu_n(C \times [0 , \tau_0])>0$. Denote by $\mu_n \in \M_{\sigma}$ the measure satisfying that $\nu_n=\frac{(\mu_n \times \text{Leb})|_{Y} }{(\mu_n \times \text{Leb})(Y)}.$
Suppose that $r \in \N$ is such that the cylinder $C \subset [r]$. Note that there exists positive real numbers $M_r, m_r$ such that $m_r \leq \tau(x) \leq M_r$, for every $x \in [r]$. In particular, for every $\mu \in \M_{\sigma}$,
$$
\frac{1}{M_r} \leq \frac{1}{\int \tau \, \d\mu} \leq \frac{1}{m_r}.
$$
Hence, since $\limsup_{n \to \infty}\nu_n(C \times [0 , \tau_0])>0$, there exists $A>0$ such that   $\limsup_{n \to \infty}\mu_n(C)>A$. 
Therefore,
$$
\int V \, \d\nu_n = \frac{\int \Delta_{V} \,\d \mu_n}{\int \tau  \, \d \mu_n} \geq \frac{A}{r M_r}.
$$
Thus,
$$
\lim_{n \to \infty} n \int V \, \d\nu_n = \infty.
$$
From equation \eqref{e_var_1} we see that for every $n \in \N$,
\begin{align*}
P^{\Phi}(g) \geq h_{\nu_n}(\Phi) + \int g \, \d \nu_n \geq P^{\Phi}(g -n V) - 
\frac{1}{n}  + n \int V \, \d \nu_n  \geq P_\infty^\Phi(g) +n \int V \, \d \nu_n  - \frac{1}{n}.
\end{align*}
Hence
$$
P(g) - n \int V \, d \nu_n \geq P_\infty^\Phi(g) - \frac{1}{n} \geq P_\infty^\Phi(g)-1.
$$
But $\lim_{n \to \infty} n \int V \, \d\nu_n = \infty$. This contradiction proves that $(\nu_n)_n$ converges to the zero measure and that
$$P^{\Phi}_{\infty}(g) = \lim_{n \to \infty}\left(h_{\nu_n}(\Phi) + \int g \, \d \nu_n\right).
$$
Finally, note that if $P^{\Phi}_{\infty}(g)=-\infty$, then, since for every sequence $(\nu_n)$ in $\M^{\Phi}$ converging to zero we have
$$\lim_{n \to \infty} \left(h_{\nu_n}(\Phi) + \int g \, \d \nu_n \right)\leq P^{\Phi}_{\infty}(g).$$
We obtain, $\lim_{n \to \infty} \left(h_{\nu_n}(\Phi) + \int g \, \d \nu_n \right)= - \infty$.
\end{proof}

\subsection{Semicontinuity properties of the pressure and SPR}

In this section, we relate the pressure at infinity to the upper
semicontinuity properties of the pressure map. More precisely, we show that
the pressure at infinity quantifies the failure of upper semicontinuity when escape of mass is allowed.
Analogous results have been obtained for countable Markov shifts in
\cite[Theorem 1]{itv} and \cite[Theorem 1.3]{v}, and for suspension flows
over countable Markov shifts, under the additional assumption of finite
topological entropy, in \cite[Theorem 1.7]{v}. 

Throughout this section, we assume that $(\Sigma,\sigma)$ is a transitive
countable Markov shift and that $\tau\colon\Sigma\to(0,\infty)$ is a roof
function bounded away from zero and with summable variations. We denote by
$(Y,\Phi)$ the associated suspension flow.

\begin{theorem}\label{thm:usc_sus}
Let $f: Y \to \R$ be such that $\Delta_f$ has summable variations and $P^{\Phi}(f) <\infty$. Let $(\nu_n)_n, \nu \in \M_{\Phi}$ and $\lambda \in [0,1]$. Assume that $(\nu_n)_n$ converges in the cylinder topology to $\lambda \nu$ and
$\liminf_{n \to \infty} \int f d \nu_n > -\infty$. Then
\begin{equation*}
\limsup_{n \to \infty} \left(h_{\nu_n}(\Phi) + \int f \, \d \nu_n\right) \leq \lambda \left(h_\nu(\Phi) + \int f \, \d \nu \right) + (1- \lambda) P^\Phi_{\infty}(f).\end{equation*}
\end{theorem}

\begin{proof}Let $(\mu_n)_n$, $\mu$ be the  invariant measures of $\Sigma$ associated with $(\nu_n)_n$ and $\nu$ respectively. It follows by Remark \ref{rem:topsus} that if $\limsup_{n\to\infty}\int \tau \, \d \mu_n=\infty$ then $(\nu_n)_n$ converges to the zero measure, in which case the result follows by Theorem \ref{prin_var_flow_infty}. We will assume that $\sup_n\int \tau \, \d\mu_n<D$ for some $D$. Maybe after passing to a subsequence we can assume that $(\mu_n)_n$ converges on cylinders to $\lambda_1\mu$ and $\lim_{n\to\infty}\int\tau \, \d\mu_n=\lambda_2\int\tau \, \d\mu$, where $\lambda=\lambda_1/\lambda_2$. Note that $\liminf_{n\to\infty} \int f \, \d\nu_n>-\infty$ iff $\liminf_{n\to\infty}\frac{\int\Delta_f \, \d\mu_n}{\int \tau \, \d\mu_n}>-\infty$ iff there exists $C$ such that $\int \Delta_f \, \d\mu_n\ge C\int \tau \, \d\mu_n$ for every $n\in\N$. We conclude that $$\liminf_{n\to \infty} \int (\Delta_f-t\tau) \, \d\mu_n\ge (C-t)\limsup_{n\to \infty} \int \tau \, \d\mu_n>-|C-t|D.$$ We can now apply the upper semicontinuity result proved in \cite[Theorem 1.3]{v} and conclude the proof of the result as in \cite[Theorem 1.7]{v}. In \cite[Lemma 8.2.1]{v} it is proved that $P^\Phi_{\infty}(f)=\inf\{t: P_\infty(\Delta_f-t\tau)\le 0\}$. In particular, if $t>P^\Phi_{\infty}(f)$, then $P_\infty(\Delta_\phi-t\tau)<0$. By \cite[Theorem 1.3]{v} we obtain the
 $$\limsup_{n\to\infty}\left(h_{\mu_{n}}(\sigma)-\int (\Delta_f-t\tau)  \, \d\mu_{n}\right)\le \lambda_1\left(h_{\mu}(\sigma)- \int (\Delta_f-t\tau) \, \d\mu\right),$$
and therefore, 
\begin{align*}
\limsup_{n\to\infty} \left(h_{\nu_n}(\Phi)-\int (\phi-t) \, \d\nu_n \right)&=\limsup_{n\to\infty}\left(\frac{h_{\mu_{n}}(\sigma)-\int (\Delta_f-t\tau)  \, \d\mu_{n}}{\int \tau \, \d\mu_{n}}\right)\\
&\le \frac{\lambda_1}{\lambda_2}\left(\frac{h_{\mu}(\sigma)- \int (\Delta_f-t\tau) \, \d\mu}{\int \tau d\mu}\right)\\
&=\lambda \left(h_\nu(\Phi)- \int (\phi-t)\, \d\nu \right),
\end{align*}
which is equivalent to 
$$\limsup_{n\to\infty} \left(h_{\nu_n}(\Phi)+\int \phi \, \d\nu_n\right)\le \lambda \left(h_{\nu}(\Phi)+\int\phi \, \d\nu\right)+(1-\lambda)t.$$
Since $t>P_\infty^\Phi(f)$ was arbitrary, we obtain the desired inequality. 
\end{proof}

We also establish a compactness result for the space of $\Phi$-invariant sub-probability measures which is analogous to \cite[Theorem 1.1]{v}.

\begin{theorem}\label{thm:comp_sus}
Let $f: Y \to \R$ be such that $\sup f < \infty$,  $\Delta_f$ has summable variations and $P^{\Phi}(f) <\infty$. Assume that $(\nu_n)_n$ is a sequence in $\M_{\Phi}$ such that $\liminf_{n \to \infty} \int f \, \d \nu_n > -\infty$. Then, $(\nu_n)_n$ has a subsequence that converges on cylinders. 
\end{theorem}
\begin{proof}
Maybe after substracting a positive constant we can assume that $\sup f<0$. Let $(\mu_n)_n$ be the associated sequence of invariant measures on $\Sigma$. By Remark \ref{rem:topsus}, if $\limsup_{n\to\infty}\int \tau \, \d\mu_n=\infty$ then $(\nu_n)_n$ has a subsequence that converges to the zero measure. We will assume that $\sup_{n\in\N}\int\tau \, \d\mu_n<D$ for some $D$. Since $\tau$ is bounded away from zero $P^\Phi(f)<\infty$ is equivalent to $P(\Delta_f-t\tau)<\infty$ for some $t$. Furthermore $\liminf \int f \, \d\nu_n>-\infty$ iff $\liminf_{n\to\infty}\frac{\int\Delta_f \, \d\mu_n}{\int \tau \, \d\mu_n}>-\infty$ iff there exists $C$ such that $\int \Delta_f \, \d\mu_n\ge C\int \tau \, \d\mu_n$ for every $n\in\N$. Note that 
$\int (\Delta_f-t\tau)\, \d\mu_n\ge (C-t)\int \tau \, \d\mu_n>-|C-t|D.$ In particular, if $\phi:=\Delta_f-t\tau$, we can apply \cite[Theorem 1.1]{v} and conclude that $(\mu_n)_n$ has a subsequence that converges on cylinders. This implies that $(\nu_n)_n$ has a convergent subsequence (see Remark \ref{rem:topsus}). \end{proof}

We  introduce the notion of strongly positive recurrent functions in the suspension flow setting (see \cite{v, irv2} for related results for finite entropy suspension flows). 

\begin{definition}
Let $(\Sigma,\sigma)$ be a  transitive countable Markov shift. Let  $\tau:\Sigma\to \R$ be a function of summable variations which is bounded away from zero. Let $(Y,\Phi)$ be the associated suspension flow. A continuous  function $f:Y \to \R$ such that $\Delta_f$ is of summable variations is \emph{strongly positive recurrent} (SPR) if $P^{\Phi}_{\infty}(f)<P^{\Phi}(f)$. 
\end{definition}

The following result establishes, under a mild additional assumption, both
the existence of equilibrium measures for strongly positive recurrent
potentials and the differentiability of the pressure.
Define 
\begin{equation}\label{s_infty}
    s_\infty(f)=\inf\{t\ge0: P^\Phi(tf)<\infty\}
\end{equation}

\begin{theorem}\label{thm:spr_deri}
Let $f: Y \to \R$ be SPR such that $\sup f < \infty$ and $P^{\Phi}(f) <\infty$. Assume that $s_{\infty}(f) <1$. Let $(\nu_n)_n$ be a sequence in $\M_{\Phi}$ such that $$\lim_{n \to \infty} \left(h_{\nu_n}(\Phi) + \int f \, \d \nu_n \right) = P^{\Phi}(f).$$
Then, there exists an equilibrium state $\nu_f$ and $\nu_n \to \nu_f$ in the weak$^*$ topology. Moreover, if $g:Y \to \R$ is continuous, bounded and $\Delta_g$ has summable variations, then the pressure function $t \mapsto P^{\Phi}(f+ tg) $ is differentiable at zero and its derivative  is $\int g \, \d \nu_f$.
\end{theorem}

\begin{proof} Let us first prove the existence of an equilibrium state. Since $s_\infty(f)<1$, there exists $\epsilon>0$ such that $P^{\Phi}((1-\epsilon)f)<\infty$. Note that
$$\lim_{n \to \infty} \left(h_{\nu_n}(\Phi) + \int f \, \d \nu_n \right)\le P^{\Phi}((1-\epsilon)f)+\epsilon\liminf_{n\to\infty} \int f \, \d\nu_n.$$
We conclude that $\liminf_{n\to\infty}\int f \, \d\nu_n>-\infty$. Because $\Delta_f$ has summable variations $f$ has at most one equilibrium state.  By Theorem \ref{thm:comp_sus} we know that any subsequence has a subsubsequence which converges on cylinders to $\lambda \nu$, for some $\nu\in\M_\Phi$ and $\lambda\in[0,1]$. It is enough to prove that $\lambda=1$ and that $\nu$ is an equilibrium state. For simplicity we will assume that $(\nu_n)_n$ converges on cylinders to $\lambda\nu$. It follows by Theorem \ref{thm:usc_sus} that 
\begin{equation*}
P^\Phi(f)=\limsup_{n \to \infty} \left(h_{\nu_n}(\Phi) + \int f \, \d \nu_n\right) \leq \lambda \left(h_\nu(\Phi) + \int f \, \d\nu \right) + (1- \lambda) P^\Phi_{\infty}(f).\end{equation*}
Since $f$ is SPR, that is, $P^\Phi_{\infty}(f)<P^\Phi(f)$, we conclude that $\lambda=1$ and that $P^\Phi(f)=h_\nu(\Phi) + \int f \, \d \nu$. In other words, $\nu$ is an equilibrium state, which by uniqueness we will denote by $\nu_f$. 

We will now prove that the derivative of   $h(t)=P^\Phi(f+tg)$ at $t=0$ is equal to $\int g \, \d\nu_f$. First note that $f+tg$ is SPR for $|t|$ sufficiently small and that $s_\infty(f+tg)=s_\infty(f)<1$. As proved above, we obtain that $f+tg$ has an equilibrium state for $|t|$ small (which we assume from now on). Denote the equilibrium state of $f+tg$ by $\nu_t$; with this notation $\nu_0=\nu_f$. Note that if $t>0$ we have that 
$$\int g \, \d\nu_0\le \frac{P^\Phi(f+tg)-P^\Phi(f)}{t}\le \int g \, \d\nu_t,$$
and similarly, if $t<0$ we have that 
$$\int g \, \d\nu_0\ge \frac{P^\Phi(f+tg)-P^\Phi(f)}{t}\ge \int g \, \d\nu_t.$$
Observe that 
$$P^\Phi(f)=\lim_{t\to 0}P^\Phi(f+tg)=\lim_{t\to 0} \left(h_{\nu_t}(\Phi)+\int (f+tg)\, \d\nu_t\right)=\lim_{t\to 0}h_{\nu_t}(\Phi)+\int f\, \d\nu_t.$$
We already proved that under this condition we have that $(\nu_t)_{t}$ converges in the weak$^*$ topology to $\nu_0$ as $t\to 0$. Combining this observation with the inequalities above we obtain that $h'(t)=\int g \, \d\nu_0$. \end{proof}

\subsection{Further properties of the pressure at infinity}
In this section, we collect basic properties of the pressure at infinity regarding its convexity, regularity and the relation between its derivatives and particular sequences of measures converging to zero.

\begin{lemma} The map $t\mapsto P_\infty^\Phi(t g)$ is convex and continuous wherever defined. Moreover, $|P_\infty^\Phi(g_1)-P_\infty^\Phi(g_2)|\le |g_1-g_2|_0$
\end{lemma}

This directly follows from the variational principle (Theorem \ref{prin_var_flow_infty}), the fact that the supremum of convex functions is convex, and the definition of the pressure at infinity.

In order to describe the asymptotics of the pressure at infinity we consider the following definition.

\begin{definition}Let $g:Y\to\R$ be a continuous function. We define $$\beta_\infty(g)=\sup_{(\nu_n)_n\to0}\limsup_{n\to\infty}\int g \, \d\nu_n \quad\text{ and }\quad \alpha_\infty(g)=\inf_{(\nu_n)_n\to0}\limsup_{n\to\infty}\int g \, \d\nu_n,$$ where the supremum and infimum runs over all sequences $(\nu_n)_n$ converging to the zero measure. 
\end{definition}

Note that $\beta_\infty(g)$ and $\alpha_\infty(g)$ are the asymptotic derivatives of the convex function $t\mapsto P_\infty^\Phi(t g)$. The proof of the following is a direct consequence of standard properties of convex functions (see \cite[p.46]{bo}).

\begin{lemma}\label{lem:asymp}Let $f$ and $g$ be bounded and continuous functions on $Y$. Then, 
$$\lim_{t\to\infty} \frac{1}{t}P_\infty^\Phi(f+tg)=\beta_\infty(g),\text{ and } \lim_{t\to-\infty} \frac{1}{t}P_\infty^\Phi(f+t g)=\alpha_\infty(g)$$
\end{lemma}

The derivatives of the pressure at infinity for suspension flows satisfy properties similar to those in the discrete time case.

\begin{proposition} \label{der.flow}
Let $(\Sigma, \sigma)$ be a transitive  CMS   and $\tau: \Sigma \to \R$ a roof function of summable variations, bounded away from zero and such that $\sup_{x\in[b]}\tau(x)< \infty$ for every $b\in\N$. Denote by $(Y, \Phi)$ the corresponding suspension flow. Let $f,g:Y\to \R$ be continuous functions such that  $\Delta_g$ and $\Delta_f$ are of summable variations. Furthermore, assume that $\sup_{x\in[b]}|\Delta_f(x)|< \infty$ and $\sup_{x\in[b]}|\Delta_g(x)|< \infty$ for every $b\in\N$. Suppose that  $Q^{\Phi}(t):=P^{\Phi}(f+ tg)$ is finite in a neighborhood of $t=0$. Set $Q_\infty^{\Phi}(t):=P_\infty(f+t g)$.

\begin{enumerate}

\item There exists a  sequence   $(\nu_n)_n$ in $\M_{\Phi}$ such that $$\lim_{n \to \infty} \left(h_{\nu_n}(\Phi)  + \int  f \, \d \nu_n\right)= P^{\Phi}(f)\quad \text{ and }\quad \lim_{n \to \infty} \int  g \, \d \nu_n=D$$ 
if and only if $D\in\partial Q^{\Phi}(0)$. Moreover, if $\nu$ is an equilibrium measure for $f$ and $g\in L^1(\nu)$ then $\int g \, \d \nu \in \partial Q^{\Phi}(0)$. 

\item There exists a  sequence    $(\nu_n)_n$ in $\M_{\Phi}$ such that $\nu_n \to 0$, $$\lim_{n \to \infty} \left(h_{\nu_n}(\Phi)  + \int  f \, \d \nu_n\right)= P_\infty^{\Phi}(f)\quad \text{ and }\quad \lim_{n \to \infty} \int  g \, \d \nu_n=D$$ 
if and only if $D\in\partial Q_\infty^{\Phi}(0)$.
    \end{enumerate}

Moreover, for parts $(a)$ and $(b)$, the sequence $(\nu_n)_n$ can be chosen such that each measure is ergodic and compactly supported.

\end{proposition}

\begin{proof}
The proofs of (a) and the right implication of (b) are exactly the same as in the  proof of Proposition \ref{der}. The left implication of (b)  follows as in 
Proposition \ref{der} but using Theorem \ref{pressure_density_susp}.
\end{proof}

The  following is a direct consequence of Proposition \ref{der.flow}.

\begin{corollary} \label{coro:der} For every $\gamma\in (\alpha_\infty(g),\beta_\infty(g))$, there exists a sequence of ergodic and compactly supported measures $(\nu_n)_n$ which converges to the zero measure and such that $\int g \, \d\mu_n\to\gamma$. 
\end{corollary}

\section{Rate functions} \label{sec:rate}

In this section, we define the rate functions utilized in our large deviation results and investigate their fundamental properties. These rate functions are formulated explicitly in terms of thermodynamic quantities. Notably, we distinguish between rate functions associated with the entire system and those at infinity; the latter are defined precisely via the pressure at infinity.

Let $(\Sigma,\sigma)$ be a transitive countable Markov shift and $\tau:\Sigma\to\R$ a roof function bounded away from zero, of summable variations and such that $\sup_{x\in[b]}\tau(x)< \infty$ for all $b\in\N$. Consider $(Y,\Phi)$  the associated suspension flow.

\begin{definition} Let $f, g$ be continuous functions defined on $Y$ such that $\Delta_f,\Delta_g$ are of summable variations. Assume that $P^\Phi(f)=0$ and $g$ is bounded. For $\gamma\in\R$ define $$I^{f,g}(\gamma)=\inf_{q\ge 0} P^\Phi(f+q(g-\gamma)) \quad \text{ and } \quad  I^{f,g}_\infty(\gamma)=\inf_{q\ge 0} P^\Phi_\infty(f+q(g-\gamma)).$$ 
\end{definition}

\begin{lemma}\label{continuity} $I^{f,g}$ is continuous on $(\alpha(g),\beta(g))$ and $I^{f,g}_\infty$ is continuous on $(\alpha_\infty(g), \beta_\infty(g))$. 
\end{lemma}

\begin{proof}
For every fixed $q\geq 0$ we have,
$ P^\Phi(f+q(g-\gamma)) =
P^\Phi(f+qg)-q\gamma,$
which is affine as a function of $\gamma$. Therefore,
\[
I^{f,g}(\gamma)= \inf_{q\geq 0}\bigl(P^\Phi(f+qg)-q\gamma\bigr)
\]
is concave.  It remains to check that this function is finite on the stated interval. Let
$\gamma\in(\alpha(g),\beta(g))$. Since $\gamma<\beta(g)$, choose an invariant measure $\nu_+$ such that $\int g\,d\nu_+>\gamma.$ By the variational principle,
\[
P^\Phi(f+qg)-q\gamma
\geq
h_{\nu_+}(\Phi)+\int f\,d\nu_+
+
q\left(\int g\,d\nu_+-\gamma\right).
\]
For $q\geq 0$, the last term is nonnegative. Hence, $I^{f,g}$ is bounded from below at $\gamma$. Since taking $q=0$ gives $I^{f,g}(\gamma)\leq P^\Phi(f)$,
we obtain $-\infty<I^{f,g}(\gamma)<+\infty$.
Thus, $I^{f,g}$ is a finite concave function on the open interval
$(\alpha(g),\beta(g))$. Therefore, it is continuous.

The same argument, replacing $P^\Phi$, $\beta(g)$ and $\alpha(g)$ by $P^\Phi_\infty$, $\beta_\infty(g)$ and $\alpha_\infty(g)$,  gives the continuity of $I^{f,g}_\infty$  on
$(\alpha_\infty(g),\beta_\infty(g))$.
\end{proof}

\begin{theorem}\label{RATEVARIATION} Let $f, g$ be continuous functions defined on $Y$. Assume that $\Delta_f,\Delta_g$ are of summable variations and that $\sup_{x\in[b]}|\Delta_f(x)|$, $\sup_{x\in[b]}|\Delta_g(x)|$ are finite for every $b\in\N$. Suppose that $P^\Phi(f)=0$ and $g$ is bounded. The following statements hold:
\begin{enumerate}
\item  If $\gamma\in (\alpha(g),\beta(g))$, then 
$$I^{f,g}(\gamma)=\sup\bigg\{h_{\nu}(\Phi)+\int f \, \d\nu:\int g \, \d\nu\ge \gamma\bigg\},$$
where the supremum runs over all ergodic and  compactly supported measures.

\item If $\gamma\in(\alpha_\infty(g),\beta_\infty(g))$, then
$$
I^{f,g}_\infty(\gamma)=\sup_{(\nu_n)_n\to0}\left\{\limsup_{n\to\infty}\left(
h_{\nu_n}(\Phi)+\int f\,d\nu_n\right):\liminf_{n\to\infty}\int g\,d\nu_n\geq\gamma
\right\},
$$
where the supremum runs over all sequences $(\nu_n)_n$ of ergodic
and compactly supported measures converging to the zero measure in
the cylinder topology.
\end{enumerate}
\end{theorem}

\begin{proof} We begin with the proof of (a). Let $\nu \in \M_{\Phi}$ be such that $\int g \, \d\nu \ge \gamma.$ Note that for every $q \geq 0$ we have
 \begin{align*}
        P^\Phi(f+q(g-\gamma))&\ge  h_{\nu}(\Phi)+\int (f+q(g-\gamma)) \,\d\nu 
        \ge  h_{\nu}(\Phi)+\int f \,\d\nu.
    \end{align*}
In particular
\begin{equation*}
I^{f,g}(\gamma)\ge\sup_{\nu \in \M_{\Phi}}\bigg\{ h_{\nu}(\Phi)+\int f \,\d\nu : \int g \, \d\nu \ge \gamma\bigg\}.
\end{equation*}
In order to prove the opposite inequality, define 
 $A(q)=P^\Phi(f+q(g-\gamma))$ and $C(q)=P^\Phi(f+q g)$.   Since $A$ is continuous  and $\lim_{q\to\infty}\frac{A(q)}{q}=\beta(g)-\gamma>0$, there exists $q_0\in [0,\infty)$ such that  $A(q_0)=I^{f,g}(\gamma)$. Indeed, the function $A(q)$ has a positive asymptotic slope, thus its infimum must be attained.  We consider  two cases, depending on whether $q_0 = 0$ or $q_0 > 0$. 
 
 Assume that $q_0>0$, in this case, we have $A'_{-}(q_0)\le 0\le A'_+(q_0)$, and therefore $C'_{-}(q_0)\le \gamma\le C'_+(q_0)$.  It follows from Proposition \ref{der.flow}    
     that there exists a sequence $(\nu_n)_n$ of ergodic and compactly supported measures in $\M_{\Phi}$ such that $\lim_{n\to\infty}\int g \, \d\nu_n=\gamma$ and 
    $$\lim_{n\to\infty} \left(h_{\nu_n}(\Phi)+\int (f + q_0(g-\gamma)) \,\d\nu_n \right)= I^{f,g}(\gamma).$$ Therefore, 
        $$\lim_{n\to\infty} \left( h_{\nu_n}(\Phi)+\int f \, \d\nu_n \right)= I^{f,g}(\gamma),$$ which proves the formula.

    Consider now the case in which $q_0=0$. We have  $A'_{+}(0)\ge 0$, and therefore $C'_{+}(0)\ge \gamma$. In particular, there exists a sequence $(\nu_n)_n$ of ergodic and compactly supported measures in $\M_{\Phi}$ such that $\lim_{n\to\infty} \left(h_{\nu_n}(\Phi)+\int f \,\d\nu_n\right)= P^\Phi(f)$ and $\liminf_{n\to\infty}\int  g \,  \d\nu_n\ge \gamma$. We conclude that 
    $$\sup_{(\nu_n)_n} \left\{\limsup_{n\to\infty} \left( h_{\nu_n}(\Phi)+\int f \, \d\nu_n \right): \liminf_{n\to\infty}\int g \, \d\nu_n\ge \gamma \right\}=P^\Phi(f).$$
    Since in this case $I^{f,g}(\gamma)=P^\Phi(f)$, we obtain the desired formula.\\

We now prove part (b). Let $(\nu_n)_n$ be a sequence of compactly supported
ergodic measures such that $\nu_n\to0$ and
$\liminf_{n\to\infty} \int g\,\d\nu_n\geq\gamma$.
Fix $q\geq0$. For every $\epsilon>0$, we have
$ \int g\,\d\nu_n-\gamma\geq-\epsilon$ for all sufficiently
large $n$. Hence, by the variational principle at infinity,
\[
\begin{aligned}
    P^\Phi_\infty\bigl(f+q(g-\gamma)\bigr)
    &\geq
    \limsup_{n\to\infty}
    \left(
        h_{\nu_n}(\Phi)+\int f\,\d\nu_n
        +q\left(\mathcal \int g\,d\nu_n-\gamma\right)
    \right)                                      \\
    &\geq
    \limsup_{n\to\infty}  \left(h_{\nu_n}(\Phi)+\int f\,\d\nu_n -q\epsilon \right).
\end{aligned}
\]
Letting $\epsilon\to0$ and then taking the infimum over
$q\geq0$, we obtain
\[
I^{f,g}_\infty(\gamma) \geq \sup_{(\nu_n)_n\to0} \left\{ \limsup_{n\to\infty}  \left(h_{\nu_n}(\Phi)+\int f\,\d\nu_n \right): \liminf_{n\to\infty}\int g\,\d\nu_n \geq\gamma \right\}.
\]
We prove the reverse inequality. Set $C_\infty(q):=P^\Phi_\infty(f+qg)$ and $A_\infty(q):=C_\infty(q)-q\gamma$. Thus $I^{f,g}_\infty(\gamma)=  \inf_{q\geq0}A_\infty(q)$.
If $I^{f,g}_\infty(\gamma)=-\infty$, the equality follows immediately from the inequality already proved. We therefore assume that $I^{f,g}_\infty(\gamma)>-\infty$. Since
\[
    I^{f,g}_\infty(\gamma)  \leq A_\infty(0) =P^\Phi_\infty(f)  \leq P^\Phi(f)=0,
\]
the number $P^\Phi_\infty(f)$ is finite. Since $g$ is bounded, $C_\infty$, and hence $A_\infty$, is a finite continuous convex function.

Suppose first that $A_\infty$ reaches its minimum at some $q_0\geq0$. If $q_0>0$, then $0\in\partial A_\infty(q_0)$, and hence $\gamma\in\partial C_\infty(q_0)$. Applying Proposition~\ref{der.flow} to the potential $f+q_0g$ and the observable $g$, we obtain a sequence $(\nu_n)_n$ of ergodic,
compactly supported measures such that $\nu_n\to0,
\int g\,\d\nu_n\to\gamma$
and
\[
    h_{\nu_n}(\Phi)+\int(f+q_0g)\,d\nu_n \rightarrow C_\infty(q_0).
\]
Therefore,
\[
\begin{aligned}
    h_{\nu_n}(\Phi)+\int f\,\d\nu_n
    =
    h_{\nu_n}(\Phi)+\int(f+q_0g)\,d\nu_n
    -q_0 \int g \, \d \nu_n                        
    \longrightarrow
    C_\infty(q_0)-q_0\gamma
    =
    A_\infty(q_0)
    =
    I^{f,g}_\infty(\gamma).
\end{aligned}
\]
Since $\int g\,\d\nu_n\to\gamma$, this sequence satisfies the required liminf constraint.

Suppose now that $q_0=0$. Since $0$ minimizes $A_\infty$ on $[0,\infty)$, we have $(A_\infty)'_+(0)\geq0$. Therefore, setting $D:=(C_\infty)'_+(0) =\gamma+(A_\infty)'_+(0)$, we have $D\geq\gamma$ and
$D\in\partial C_\infty(0)$.  The Proposition~\ref{der.flow} gives a
sequence $(\nu_n)_n$ of compactly supported ergodic measures such
that $\nu_n\to0$, $\int g\,\d\nu_n\to D\geq\gamma$,
and
\[
   h_{\nu_n}(\Phi)+\int f\,d\nu_n
    \longrightarrow C_\infty(0)=A_\infty(0) =I^{f,g}_\infty(\gamma).
\]
This again gives the desired sequence.
\end{proof}

\begin{corollary}\label{awayfromzero}Let $f, g$ be continuous functions defined on $Y$ with $\sup f < \infty$. Assume that $\Delta_f,\Delta_g$ are of summable variations and that $\sup_{x\in[b]}|\Delta_f(x)|$, $\sup_{x\in[b]}|\Delta_g(x)|$ are finite for every $b\in\N$. Suppose that $g$ is bounded, $P^\Phi(f)=0$ and $P^{\Phi}_{\infty}(f)<0$. 
\begin{enumerate}
\item Suppose that $s_\infty(f)<1$. Let $\nu_f$ be the equilibrium state of $f$. If $\gamma\in \left(\int g \, \d\nu_{f},\beta(g)\right)$, then $I^{f,g}(\gamma)<0$. 
\item If $\gamma<\beta_{\infty}(g)$ then $I^{f,g}_\infty(\gamma)<0$.
\end{enumerate}
\end{corollary}

\begin{proof} We prove each claim separately.
 \begin{enumerate}
\item By the proof of Theorem \ref{RATEVARIATION}, there exists $q_0 \geq 0$ such that $P^\Phi(f+q_0(g-\gamma))=I^{f,g}(\gamma)$. Since $\gamma\in \left(\int g \, \d\nu_{f},\beta(g)\right)$ we have that the right derivative of $P^\Phi(f+q(g-\gamma))$ at $q=0$ is negative (see Theorem \ref{thm:spr_deri}) and since the asymptotic slope is positive we have $q_0>0$.
Since  $P^\Phi(f)=0$ we obtain  $I^{f,g}(\gamma)=P^\Phi(f+q_0(g-\gamma))<0$. 
\item Immediate consequence of Theorem \ref{RATEVARIATION}(b) combined with the variational principle at infinity (see Theorem \ref{prin_var_flow_infty}). Note that $I^{f,g}_\infty(\gamma)\leq P_{\infty}^{\Phi}(f)$.
\end{enumerate}
\end{proof}

\begin{remark}\label{rem:sharp}
The assumption that $P^{\Phi}_{\infty}(f)<P^{\Phi}(f)$ is necessary for Corollary \ref{awayfromzero}(a) to hold. Given that $P^{\Phi}_{\infty}(f)=P^{\Phi}(f)=0$ we can take $i\in \N$ and consider  $g=-\chi_{([i]\times \R)\cap Y_0}$. Note that $\int g\,\d\nu_f<0$. We now take $\gamma\in \left(\int g\,\d\nu_f,0\right)$ and since $P^{\Phi}_{\infty}(f)=0$ we can find a sequence of compactly supported ergodic measures $(\nu_n)_n$ such that $\lim_{n\to\infty}\int g\,\d\nu_n=0>\gamma$ and $\lim_{n\to\infty} (h_{\nu_n}(\Phi)+\int f\,\d\nu_n)=0$. Therefore $I^{f,g}(\gamma)=0$  for all $\gamma\in \left(\int g\,\d\nu_f,0\right)$ by Theorem \ref{RATEVARIATION}$(a)$.
\end{remark}

\section{Large Deviations} \label{sec:ldp}
In this section, we prove level-1 large deviation results for suspension flows over countable Markov shifts, extending classical estimates from the compact setting to this non-compact framework. Furthermore, we obtain a large deviation principle at infinity by considering sets of points that diverge on average at some given scale. For these diverging trajectories, the associated rate function is governed by the pressure at infinity rather than the classical topological pressure.

Let $(\Sigma,\sigma)$ be a topologically mixing countable Markov shift. Consider a roof function  $\tau:\Sigma\to\R$ which is bounded away from zero, of summable variations and such that $\sup_{x\in[b]}\tau(x)< \infty$ for all $b\in\N$. Let $(Y,\Phi)$ be the associated suspension flow. Recall that $Y_0= \{ (x,t)\in \Sigma  \times \R \colon 0 \le t <\tau(x)\}$.  For a fixed $m\in\N$ recall we set  $\tau_m=\sup\{\tau(x):x_1\leq m\}$.

\begin{definition} \label{def:Rt} Fix $m, Q, M \in\N$.  Set $K_m= \bigcup_{i=1}^m [i]$ and $\widetilde{K}_m=(K_m\times\R)\cap Y_0$. Define $$R_T(m)=\widetilde{K}_m\cap \Phi^{-T} \big(\widetilde{K}_m\big),$$
and 
$$R_T(m;Q,M)=\left\{(x,t)\in R_T(m): \frac{1}{T}\int_0^T \chi_{\widetilde{K}_m}(\Phi_s((x,t)))\,\d s <\frac{1}{Q}\right\}.$$
For a continuous function $g:Y\to\R$ and $\gamma\in\R$ we define
$$Y^g_{T}(\gamma)=\bigg\{(x,t)\in R_T(m): \frac{1}{T}\int_0^T g(\Phi_s((x,t))\,d s> \gamma \bigg\},$$
and
$$Y^g_{T}(\gamma;Q,M)=\bigg\{(x,t)\in R_T(m; Q,M): \frac{1}{T}\int_0^T g(\Phi_s((x,t))\,\d s> \gamma \bigg\}.$$
\end{definition}

\begin{remark}\label{rem:n_T} Let $(x,t)\in R_T(m)$, where $x=(x_1,x_2,\ldots)$. Define $n_T(x)\in\N$ such that $\Phi_T((x,t))\in([x_{n_T(x)}]\times \R)\cap Y_0$. Note that $n_T(x)-1$ is the number of times the orbit of $(x,t)$ returns to the cross section $\Sigma\times\{0\}$ up to time $T$. Moreover, $x_1,x_{n_T(x)}\in\{1,\ldots,m\}$ and $$|S_{n_T(x)}\tau(x)-T|\le 2\tau_m.$$
Similarly, $$\left|\int_0^Tg(\Phi_s((x,t))\,\d s-S_{n_T(x)}\Delta_g(x)\right|\le 2\tau_m\|g\|_0,$$
for every $g:Y\to\R$ bounded.
\end{remark}

The proof of Theorem \ref{thm:ldp} is organized as follows.  In Section \ref{sec:upperbound} we establish the large deviation upper bounds for parts (a) and (b). Then, in Section \ref{sec:lowerbound}, we establish the large deviation lower bounds and complete the proof. For the remainder of this section we assume the hypotheses of Theorem \ref{thm:ldp} regarding  $(\Sigma,\sigma)$, $(Y,\Phi)$ and  the functions $f,g, \tau$.

\subsection{Proof of the upper bounds}\label{sec:upperbound} We start with a simple result.

\begin{lemma}\label{lem:boundupp}There exist $N\in\N$ and $C>0$ such that for all $x=(x_1,x_2,\ldots)\in\Sigma$ with $x_1,x_n\in\{1,\ldots,m\}$, there exist $\omega_1,\omega_2\in \N^N$ such that $(1,\omega_1,x_1,\ldots,x_n,\omega_2,1)$ is admissible and for all $z\in [1,\omega_1,x_1,\ldots,x_{n},\omega_2,1]$ we have that 
$$\max\{|S_{n+2N+1}\tau(z)-S_{n}\tau(x)|,
|S_{n+2N+1}\Delta_{h}(z)- S_{n}\Delta_{h}(x)|\}\le C,$$
where $h\in\{f,g\}$.
\end{lemma}
\begin{proof} Since $(\Sigma,\sigma)$ is topologically mixing, there exists $N$ such that for every $i,j\in\{1,\ldots,m\}$, there are admissible words $w_1,w_2$ of length $N$ such that  $(1,w_1,i)$ and $(j,w_2,1)$ are admissible. Since $i,j$ runs over a finite set, the collection of words $\{\omega_1,\omega_2\}$ is finite, and therefore they use only finitely many letters in the alphabet, say $\{1,\ldots,r\}$. Let  $C_1=\sup_{x\in K_r}\Delta_g(x)$. In particular, if $(x,t)\in R_T(m)$, then $x_1,x_{n_T(x)}\in \{1,\ldots,m\}$, and therefore there exists a non-empty cylinder $[1,\omega_1,x_1,\ldots,x_{n_T(x)},\omega_2,1]$, where $\omega_1,\omega_2\in\N^N$. Let $z\in [1,\omega_1,x_1,\ldots,x_{n_T(x)},\omega_2,1]$. Furthermore, $$|S_{n_T(x)+2N+1}\tau(z)- S_{n_T(x)}\tau(x)|\le (2N+1)\tau_m,$$
and $$|S_{n_T(x)+2N+1}\Delta_{g}(z)- S_{n_T(x)}\Delta_{g}(x)|\le (2N+1)C_1.$$
Finally, take $C=(2N+1)\max\{C_1,\tau_m\}.$
\end{proof}

Before explaining the covering argument we make some general observations. Fix $m\in\N$ and $(x,t)\in R_T(m)$. If follows by Remark \ref{rem:n_T} that 
\begin{align}\label{eq:up1} S_{n_T(x)}(\Delta_g-\gamma\tau)(x)\ge \int_0^T g(\Phi_s((x,t))\,\d s-\gamma T- 2\tau_m(\|g\|_0+|\gamma|).\end{align}
 By Lemma \ref{lem:boundupp}, there exists a periodic point $$z_x:=(\overline{1,\omega_1,x_1,\ldots,x_{n_T(x)},\omega_2})\in [1,\omega_1,x_1,\ldots,x_{n_T(x)},\omega_2,1]$$  such that for some $C>0$ we have that \begin{align}\label{eq:up2}
     S_{n_T(x)+2N+1}(\Delta_g-\gamma\tau)(z_x)\ge S_{n_T(x)}(\Delta_g-\gamma\tau)(x)-C(1+|\gamma|).
 \end{align}

Assume that $(x,t)\in Y^g_{T}(\gamma)\subset R_T(m)$. In particular, $\int_0^T g(\Phi_s((x,t))\, \d s > \gamma T$. It follows by (\ref{eq:up1}) and (\ref{eq:up2}) that if $E:=C(1+|\gamma|)+2\tau_m(\|g\|_0+|\gamma|)$, then 
\begin{align}\label{ineq:z_x}
    e^{S_{n+2N+1}\Delta_f(z_x)}\le e^{qE} e^{S_{n_T(x)+2N+1}(\Delta_f+q(\Delta_g-\gamma \tau))(z_x)},
\end{align}
for every $q\ge 0$.

Define the cylinder  $C_x=[x_1,\ldots,x_{n_T(x)}]$. By the weak Gibbs property we have that $$\mu_{\Delta_f}(C_x)\le D_1e^{S_{n_T(x)}\Delta_f(x)}$$ for some $D_1>0$ (see Lemma \ref{weak_gibbs}). Moreover, by Lemma \ref{lem:boundupp}, there exists $D_2>0$ such that 
\begin{align}\label{eq:up3}\mu_{\Delta_f}(C_x)\le D_2e^{S_{n_T(x)+2N+1}\Delta_f(z_x)}.
\end{align}  
By inequality (\ref{ineq:z_x}) we conclude that for some $D_3>0$ we have that 
\begin{align}\label{eq:up4}\mu_{\Delta_f}(C_x)\le D_3 e^{S_{n_T(x)+2N+1}(\Delta_f+q(\Delta_g-\gamma \tau))(z_x)}.
\end{align}

By construction $(z_x,0)\in Y$ is a periodic point of length $s_x:=S_{n_T(x)+2N+1}\tau(z_x)$. By Lemma \ref{lem:boundupp} and Remark \ref{rem:n_T} we have that 
$$|S_{n_T(x)+2N+1}\tau(z_x)-T|\le |S_{n_T(x)+2N+1}\tau(z_x)-S_{n_T(x)}\tau(x)|+|S_{n_T(x)}\tau(x)-T|\le C+R_m=:C_0,$$
and therefore $$(z_x;s_x)\in \mathrm{FPer}_1(T-C_0,T+C_0).$$

\begin{proof}[Proof of upper bounds for Theorem \ref{thm:ldp}] 
We start with the proof of the upper bound for part (a). 
Note that 
$Y^g_{T}(\gamma)\subset \bigcup_{(x,t)\in Y^g_{T}(\gamma)}(C_x\times [0,\tau_m]),$ 
and therefore, 
\begin{align*}\nu_f(Y^g_{T}(\gamma))&\le \left(\int\tau \, \d\mu_f\right)^{-1}\tau_m\sum_{(x,t)\in Y^g_{T}(\gamma)}\mu_{\Delta_f}(C_x) \\
&\le  \left(\int\tau \, \d\mu_f\right)^{-1} \tau_m D_3 \sum_{(z_x;s_x)\in \mathrm{FPer}_1(T-C_0,T+C_0)} e^{S_{n_T(x)+2N+1}(\Delta_f+q(\Delta_g-\gamma\tau))(z_x)}\\
&=\left(\int\tau \, \d\mu_f\right)^{-1} \tau_m D_3 \sum_{(z_x;s_x)\in \mathrm{FPer}_1(T-C_0,T+C_0)} \exp\bigg(\int_0^{s_x} (f+q(g-\gamma))(\Phi_t(z_x,0)) \,\d t\bigg),
\end{align*}
where we used inequality (\ref{eq:up4}). Taking $\frac{1}{T}\log(\cdot)$ and limsup as $T\to\infty$ we obtain that
$$\limsup_{T\to\infty}\frac{1}{T}\log\nu_f(Y^g_{T}(\gamma))\le P^{\Phi}(f+q(g-\gamma))$$
for every $q\ge 0$. Thus we obtain that
$$\lim_{T\to\infty}\frac{1}{T}\log\nu_f\left( \bigg\{(x,t)\in R_T(m): \frac{1}{T}\int_0^T g(\Phi_s((x,t))\,d s> \gamma \bigg\}\right)\le I^{f,g}(\gamma),$$
which is the upper bound for part (a).

We now turn to part (b). Let $(x,t)\in R_T(m;Q,M)$.  
Note that $s_x\in (T-C_0, T+C_0)$, and that
$$\int_0^{s_x} \chi_{\widetilde{K}}(\Phi_s((z_x,0))) \, \d s <\int_0^T \chi_{\widetilde{K}_M}(\Phi_s((x,t))) \,\d s +2C_0\tau_m.$$
Suppose $T$ is large enough so that $T-C_0>\frac{2}{3}T$. In particular, we have that
\begin{align*}
    \frac{1}{s_x}\int_0^{s_x} \chi_{\widetilde{K}_M}(\Phi_s((z_x,0)))\,\d s&< \frac{3}{2T}\int_0^T \chi_{\widetilde{K}_M}(\Phi_s((x,t)))\,\d s+\frac{3C_0\tau_m}{T}\\
    &<\frac{3}{2Q}+\frac{3C_0\tau_m}{T}
\end{align*}
Define $Q'=\big(\frac{3}{2Q}+\frac{3C_0\tau_m}{T}\big)^{-1}$. If $T$ is large enough we can further assume that $Q'>\frac{Q}{2}$, and therefore $$(z_x;s_x)\in \mathrm{FPer}_1(T-C_0, T+C_0; \frac{Q}{2},M).$$

Observe that 
$$Y^g_{T}(\gamma;Q,M)\subset \bigcup_{(x,t)\in Y^g_{T}(\gamma;Q,M)}(C_x\times [0,\tau_m]),$$
and therefore, 
\begin{align*}\nu_f(Y^g_{T}(\gamma;Q,M))  & \le \left(\int\tau \,\d\mu_f\right)^{-1}\tau_m\sum_{(x,t)\in Y^g_{T}(\gamma;Q,M)}\mu_f(C_x) \\
&\le  \left(\int\tau \, \d\mu_f\right)^{-1} \tau_mD_3 \sum_{(z_x;s_x)\in \mathrm{FPer}_1(T-C_0,T+C_0;{Q}/{2},M)} e^{S_{n_T(x)+2N+1}(\Delta_f+q(\Delta_g-\gamma\tau))(z_x)}\\
&=\left(\int\tau \, \d\mu_f\right)^{-1} \tau_mD_3 \sum_{(z_x;s_x)\in \mathrm{FPer}_1(T-C_0,T+C_0;{Q}/{2},M)} \exp\bigg(\int_0^{s_x} (f+q(g-\gamma))(\Phi_t(z_x,0)) \, \d t\bigg),
\end{align*}
where we used inequality (\ref{eq:up4}). Taking $\frac{1}{T}\log(\cdot)$ and limit as $T\to\infty$ we obtain that
$$\limsup_{T\to\infty}\frac{1}{T}\log\nu_f(Y^g_{T}(\gamma;Q,M))\le P_\infty^{\Phi}(f+q(g-\gamma),1;{Q}/{2},M),$$
for every $q\ge 0$. In particular, it follows that 
$$\limsup_{T\to\infty}\frac{1}{T}\log\nu_f(Y^g_{T}(\gamma;Q,M))\le \inf_{q\ge0}P_\infty^{\Phi}(f+q(g-\gamma),1;{Q}/{2},M).$$ 
Finally, note that 
\begin{equation}\label{upper_infty}
\lim_{Q,M\to\infty}\inf_{q\ge0}P_\infty^{\Phi}(f+q(g-\gamma),1;{Q}/{2},M)=\inf_{q\ge 0}P_\infty^\Phi(f+q(g-\gamma))=I^{f,g}_\infty(\gamma).
\end{equation}
This concludes the proof of the upper bound of  Theorem \ref{thm:ldp} part (b). Part (c) also follows from equation (\ref{upper_infty}).

\end{proof}

\subsection{Proof of the lower bounds}\label{sec:lowerbound}
For a suspension flow it is not always the case that the base being mixing implies that the flow will be mixing. However we will have the following mixing type property which we will be useful for proving the lower bound.

\begin{lemma}\label{mixingcond}
There exists $m\in\N$ and $A\subseteq \{1,\ldots,m\}$ with the following properties:
\begin{enumerate}
\item
For any admissible word $(\omega_1,\ldots,\omega_k)$ where $\omega_1,\omega_k\in A$ there exists $i,j\in A$ such that
$(i,\omega_1,\ldots,\omega_k,j)$ is an admissible word.
\item 
For any $\ell\in\N$ there exist $N\in\N$ such that the following holds: for any $k\in\N$ and admissible word $(\omega_1,\ldots,\omega_{k})\in\{1,\ldots,\ell\}^{k}$, there exists an admissible word $(\eta_1,\ldots,\eta_{k'})$ where $k\le k'\leq k+N$, there exists $1\leq i\leq k'-k+1$ where $(\eta_i,\ldots,\eta_{i+k-1})=(\omega_1,\ldots,\omega_{k})$, and $\eta_1,\eta_{k'}\in A$.  
\item
For any $N\in\N$ there exists $C(N)>0$ such that for any admissible word $(\omega_1,\ldots,\omega_k)$ where $\omega_1,\omega_k\in A$
we have
$$\mu_{\Delta_f}(\{[\omega_1,\ldots,\omega_k,j_1,\ldots,j_N]:j_1,\ldots,j_N\in A\})\geq C(N)\mu_{\Delta_f}([\omega_1,\ldots,\omega_k]).$$
\end{enumerate}
\end{lemma}
\begin{proof}
Let $(a_1,a_2,\ldots,a_n)$ be an admissible word with $a_1=a_n=1$, which exists because the shift space is mixing. For part (a) it is enough to set $A=\{a_i:1\leq i\leq n-1\}$ and $m=\max A$. Part (b)  follows from the fact that $A$ and $\{1,\ldots,\ell\}$ are both finite sets and $(\Sigma,\sigma)$ is topologically mixing property. Part (c) follows by the weak Gibbs property of $\mu_{\Delta_f}$ (see Lemma \ref{weak_gibbs}). 
\end{proof}

The key to proving the lower bound of both part (a) and (b) of Theorem \ref{thm:ldp} is the following proposition. Recall that  $K_m= \bigcup_{i=1}^m [i]$ and $\widetilde{K}_m=(K_m\times\R)\cap Y_0$. 

\begin{proposition}\label{keyprop}
Let $\mu \in \M_{\sigma}$ be an ergodic measure with $\supp(\mu)\subseteq \{1,\ldots,\ell\}^{\N}$ for some $\ell\in\N$. Let $\nu=\mu\times \text{Leb}|_{Y}/(\mu\times \text{Leb})(Y)$ be the associated measure in $Y$. Assume that  $\gamma<\int g \,\d\nu$, and that $\nu(\widetilde{K}_M)<\frac{1}{Q}$. Then, we have that
$$\liminf_{T\to\infty}\frac{1}{T}\log\nu_{f}(Y^g_{T}(\gamma))\geq \liminf_{T\to\infty} \frac{1}{T}\log\nu_{f}(Y^g_{T}(\gamma;Q,M))\geq h_{\nu}(\Phi)+\int f\,\d\nu.$$
\end{proposition}

\begin{proof}
Since $Y^g_T(\gamma,Q,M)\subset Y^g_T(\gamma)$, it is enough to prove the bound $$\liminf_{T\to\infty} \frac{1}{T}\log\nu_{f}(Y^g_{T}(\gamma;Q,M))\geq h_{\nu}(\Phi)+\int f\d\nu,$$
for every $\nu$ satisfying the assumptions in the statement.

We can now build the cylinder sets that will give the lower bound. By the Birkhoff Ergodic Theorem and the Shannon-Macmillan Brieman Theorem we have that for $\mu$-almost all $\omega\in\Sigma$,
\begin{enumerate}
\item
$\lim_{n\to\infty}n^{-1}S_n\Delta_{f}(\omega)=\int\Delta_{f}\, \d\mu.$
\item 
$\lim_{n\to\infty}n^{-1}S_n\Delta_{g}(\omega)=\int\Delta_{g}\, \d\mu.$
\item 
$\lim_{n\to\infty}n^{-1}S_n\tau(\omega)=\int\tau\, \d\mu.$
\item
$\lim_{n\to\infty}n^{-1}\log(\mu([\omega_1,\ldots,\omega_n]))=h_{\mu}(\sigma).$
\item
$\lim_{n\to\infty}\frac{\sum_{i=0}^{n-1}(\tau\chi_{K_M})(\sigma^{i}(\omega))}{\sum_{i=0}^{n-1}\tau(\sigma^{i}(\omega))}<\frac{1}{Q}.$

\end{enumerate}
This means that, by Egorov's Theorem, we can find a Borel set $B\subset\{1,\ldots,\ell\}^{\N}$ where the convergence in all the above statements is uniform and $\mu(B)>\frac{1}{2}$.  Let $\min\{\int g\,\d\nu-\gamma,\frac{1}{Q}-\int\chi_{\widetilde{K}_M}\,\d\nu\}>2\epsilon>0$. We can find $N\in\N$ such that for any $n\geq N$ and $\omega\in B$:
\begin{enumerate}
\item[(a')] $n^{-1}S_n\Delta_{f}(\omega)\in \left(\int\Delta_{f} \, \d\mu-\epsilon,\int\Delta_{f} \, \d\mu+\epsilon\right).$
\item[(b')] $n^{-1}S_n\tau(\omega)\in \left(\int\tau\, d\mu-\epsilon,\int\tau\, \d\mu+\epsilon\right).$
\item[(c')] $\frac{S_n\Delta_{g}(\omega)}{S_n\tau(\omega)}>\gamma+2\epsilon.$
\item[(d')] $\mu([\omega_1,\ldots,\omega_n])\in (e^{-n(h_{\mu}(\sigma)+\epsilon)},e^{-n(h_{\mu}(\sigma)-\epsilon)}).$
\item[(e')] $\frac{\sum_{i=0}^{n-1}(\tau\chi_{K_M})(\sigma^{i}(\omega))}{\sum_{i=0}^{n-1}\tau(\sigma^{i}(\omega))}<\frac{1}{Q}-2\epsilon.$
\end{enumerate}

For each $\omega\in B$ we set
$$m_T(\omega):=\max\left\{m:\sup_{\eta\in [\omega_1,\ldots,\omega_m]}S_m\tau(\eta)\leq T-2\delta_1 \right\}.$$ Fix $T$ sufficiently large such that $m_T(\omega)\ge N$ for every $\omega\in B$ (note that $\tau$ is bounded on $B$). Observe that $T-2\delta_1\le S_{m_T(w)+1}\tau(\omega)$. By part  (b) above $S_{m_T(w)+1}\tau(\omega)\le(\int \tau \,\d\mu+\epsilon)(m_T(\omega)+1)$, and therefore $$m_T(\omega)\ge \frac{T-\delta_1'}{\int \tau \,\d\mu+\epsilon},$$ where $\delta_1'=2\delta_1-(\int \tau \,\d\mu+\epsilon).$ For each $\omega\in B$, consider the cylinder $$C(\omega,T)=[\omega_1,\ldots,\omega_{m_T(\omega)}].$$ 
Define $\mathcal B_T=\{C(\omega,T): \omega\in B\}$. 
Since $\mu(B)>1/2$ we have that
\begin{eqnarray*}
\frac{1}{2}&\leq& \sum_{C\in \mathcal B_T} \mu(C)\\
&\leq& \sum_{C\in \mathcal B_T}e^{-|C|(h_{\mu}(\sigma)-\epsilon)}\\
&\leq&  \sum_{C\in \mathcal B_T}e^{-\frac{T-\delta_1'}{\int \tau \, \d \mu+\epsilon}(h_{\mu}(\sigma)-\epsilon)}\\
&=&\#\mathcal B_T e^{-\frac{T-\delta_1'}{\int \tau \, \d \mu+\epsilon}(h_{\mu}(\sigma)-\epsilon)}.
\end{eqnarray*}
Thus
\begin{align}\label{bT}
    \#\mathcal B_T\geq \kappa e^{\frac{T}{\int \tau \, \d\mu+\epsilon}(h_{\mu}(\sigma)-\epsilon)}.
\end{align}
where $\kappa=\frac{1}{2}e^{\frac{-\delta_1'}{\int \tau \, \d\mu+\epsilon}(h_{\mu}(\sigma)-\epsilon)}.$

In order to prove the proposition we have the following claim:

\begin{claim}\label{claim:1}
There exist $N_1,N_2\in\N$ and $E>0$ such that for any $T$ sufficiently large and for each $C=[\omega_1,\ldots,\omega_n]\in\mathcal{B}_T$  there exists a cylinder set $\eta(C):=[\eta_1,\ldots,\eta_{n+N_1+N_2}]$ such that $(\eta_{N_1+1},\ldots,\eta_{N_1+n})=(\omega_1,\ldots,\omega_n)$ and 
\begin{enumerate}
\item
$([\eta_1,\ldots,\eta_{n+N_1+N_2}]\times\R)\cap Y_0\subset R_T(m)$
\item 
For any $(\eta,t)\in ([\eta_1,\ldots,\eta_{n+N_1+N_2}]\times\R)\cap Y_0$ we have
$$\int_0^T g(\Phi^s(\eta,t))\,\d s\geq  \gamma T$$
and
$$\int_0^T  \chi_{\widetilde{K}_M}(\Phi^s(\eta,t))\,\d s<\frac{1}{Q}$$
\item 
$\nu_f(([\eta_1,\ldots,\eta_{n+N_1+N_2}]\times\R)\cap Y_0)\geq E\nu_f(([\omega_1,\ldots,\omega_n]\times\R)\cap Y_0
).$
\end{enumerate}
\end{claim}

\begin{proof}
 
 We fix  $\delta_2=\max\{\sum_{n=1}^{\infty}\text{var}_n(\tau),\sum_{n=1}^{\infty}\text{var}_n(\Delta_ f),\sum_{n=1}^{\infty}\text{var}_n(\Delta_ g)\}$.  We now fix $T>0$.  We now take $[\omega_1,\ldots,\omega_n]\in\mathcal{B}_T$. We have that for any $\eta\in [\omega_1,\ldots,\omega_n]$ that
 $$T-2\delta_1\geq S_n\tau(\eta)\geq T-2\delta_1-\delta_2-C'$$
Since $\omega_1\leq l$, $\omega_n\leq l$ and the shift is topologically mixing we can find $N_1\in\N$ and an admissible cylinder
$[i_1,\ldots,i_{N_1},\omega_1,\ldots,\omega_n,i_{N_1+1},\ldots, i_{2N_1}]$ such that $i_1,i_{2N_1}\in A$ and for all $\eta\in [i_1,\ldots,i_{N_1},\omega_1,\ldots,\omega_n,i_{N_1+1},\ldots, i_{2N_1}]$ we have that 
 $$ T-2\delta_1-\delta_2-C'\leq S_{n+2N_1}\tau(\eta)< T.$$
Now using the properties of $A$ there exists $N_2'$ such that we can find $(j_1,\ldots,j_{N_2'})\in A^{N_2'}$ such that  
$ [i_{1},\ldots,i_{N_1},\omega_1,\ldots,\omega_n,i_{N_1+1},\ldots, i_{2N_1},j_1,\ldots,j_{N_2'}]$ is an admissible cylinder and for all $\eta \in [i_{1},\ldots,i_{N_1},\omega_1,\ldots,\omega_n,i_{N_1+1},\ldots, i_{2N_1},j_1,\ldots,j_{N_2'}]$ we have that $S_{n+2N_1+N_2'}\tau(\eta)\geq T+\max_{\underline{i}:i_1\in A}\{\tau(\underline{i})\}$. 
Set 
$[\eta_1,\ldots,\eta_{n+N_1+N_2}]=[i_{1},\ldots,i_{N_1},\omega_1,\ldots,\omega_n,i_{N_1+1},\ldots, i_{2N_1},j_1,\ldots,j_{N_2'}]$, where $N_2=N_1+N_2'$.
Thus
$$([\eta_1,\ldots,\eta_{n+N_1+N_2}]\times\R)\cap Y_0\subset R_T(m).$$
This completes the proof of the first part of the Lemma.

For the second part we use that for $\eta\in [\omega_1,\ldots,\omega_n]$ we have
$$\frac{S_n\Delta_g(\eta)}{S_n\tau(\eta)}>\gamma+\epsilon, \quad \text{ and } \quad \frac{\sum_{i=0}^{n-1}(\tau\chi_{K_M})(\sigma^{i}(\omega))}{\sum_{i=0}^{n-1}\tau(\sigma^{i}(\omega))}<\frac{1}{Q}-\epsilon$$
(using our assumption on $\mathcal{B}_T$ and bounded distortion). It then follows that for any $(\eta,t)\in ([\eta_1,\ldots,\eta_{n+N_1+N_2}]\times\R)\cap Y_0$ we have that
$$\int_0^T g(\Phi^s(\eta,t))\,\d s\geq (T-2\delta_1-\delta_2-C')(\gamma+\epsilon)-\mathit{O}(1)$$
and
$$\int_0^T \chi_{\widetilde{K}_M}(\Phi^s(\eta,t))\,\d s\leq  (T-2\delta_1-\delta_2-C')\left(\frac{1}{Q}-\epsilon\right)+(2\delta_1+\delta_2+C'),$$
where we have used that $g$ is bounded and that $N_1,N_2$ are independent of $T$.

The independence of $N_1,N_2$ on $T$ and the fact that $\nu_f$ projects to a weak Gibbs measure on $\Sigma$ completes the proof of part (c). \end{proof}

Assume $T$ is sufficiently large. It is clear from the construction that the map $\eta:\B_T\to \eta(\B_T)$ is injective, thus using Claim \ref{claim:1} we obtain that

\begin{eqnarray*}
\nu_{f}(\{(\omega,t)\in R_{T}(m;Q,M):\frac{1}{T}\int_{0}^T g(\Phi_s((x,t))) \,\d s>\gamma\}) 
&\geq& \sum_{C\in \B_T}\nu_{f}((\eta(C)\times\R)\cap Y_0)\\
&\gtrsim & \sum_{C\in \B_T}\nu_{f}((C\times\R)\cap Y_0)\\
&\gtrsim &  \frac{1}{\int \tau \,\d\mu_f} \sum_{C\in \B_T}e^{S_{|C|}\Delta_f(\chi_C)}\\
&\gtrsim &  e^{\frac{T(h_{\mu}(\sigma)+\int \Delta_f\, \d\mu-2\epsilon)}{\int\tau\, \d\mu+\epsilon}}  
\end{eqnarray*}
where $\chi_C$ is a point in $C$ and in the last inequality we have used bound (\ref{bT}) and (a'). Thus putting this together we get that
$$\liminf_{T\to\infty}\frac{1}{T}\log\nu_{f}(\{(\omega,t)\in R_{T}(m,Q,M):\frac{1}{T}\int_{0}^T g(\Phi_s((x,t)))\, \d s>\gamma\})\geq\frac{h_{\mu}(\sigma)+\int \Delta_f \, \d\mu-2\epsilon}{\int\tau\, \d\mu+\epsilon}$$
and letting $\epsilon\to 0$ gives
\begin{eqnarray*}
\liminf_{T\to\infty}\frac{1}{T}\log\nu_{f}(\{(\omega,t)\in R_{T}(m,Q,M):\frac{1}{T}\int_{0}^Tg(\Phi_s((x,t)))\, \d s>\gamma\})&\geq& \frac{h_{\mu}(\sigma)}{\int\tau\, \d\mu}+\frac{\int\Delta_{f}\, \d\mu}{\int\tau\, \d\mu}\\
&=& h_{\nu}(\Phi)+\int\phi\,\d\nu.
\end{eqnarray*}
This completes the proof of Proposition \ref{keyprop}.
\end{proof}

\begin{proof}[Proof of the lower bounds for Theorem \ref{thm:ldp}]
We start with part (a).
 By Theorem \ref{RATEVARIATION} there exists an ergodic $\mu$ with $\supp(\mu)\subseteq \{1,\ldots,\ell\}^{\N}$ for some $\ell\in\N$ and $\nu=\mu\times \text{Leb}|_{Y}/(\mu\times \text{Leb})(Y)$ where $\int g \, \d\nu>\gamma$, $\nu\big(([1]\times\R)\cap Y_0\big)<1$ and
$$h_{\nu}(\Phi)+\int f \, \d\nu>I^{f,g}(\gamma)-\epsilon.$$
Thus applying Proposition \ref{keyprop} (with $Q=M=1$) we get that
$$\liminf_{T\to\infty}\frac{1}{T}\log\nu_{f}( Y_{T}(\gamma))\geq \liminf_{T\to\infty} \frac{1}{T}\log\nu_{f}( Y_{T}(\gamma,1,1))\geq h_{\nu}(\Phi)+\int f \, \d\nu\geq I^{f,g}(\gamma)-\epsilon.$$
Since the choice of $\epsilon$ was arbitrary, we obtain that
$$\liminf_{T\to\infty}\frac{1}{T}\log\nu_{f}( Y_{T}(\gamma))\geq I(\gamma).$$

For part $(b)$ we fix $\gamma<\beta_{\infty}(g)$, $\epsilon>0$ and $q,M\in\N$. By Theorem \ref{RATEVARIATION} there exists an ergodic $\mu$ with $\supp(\mu)\subseteq \{1,\ldots,l\}^{\N}$ for some $l\in\N$ and $\nu=\mu\times \text{Leb}|_{Y}/(\mu\times \text{Leb})(Y)$ where $\int g \,\d\nu>\gamma$, $\nu((K_M\times\R)\cap Y_0)<\frac{1}{Q}$ and  
$$h_{\nu}(\Phi)+\int f \, \d\nu>I^{f,g}_{\infty}(\gamma)-\epsilon.$$
 Thus applying Proposition \ref{keyprop} we get that
 $$\liminf_{T\to\infty} \frac{1}{T}\log\nu_{f}( Y_{T}(\gamma,Q,M))\geq I^{f,g}_{\infty}(\gamma)-\epsilon.$$
 Thus,
$$\liminf_{T\to\infty} \frac{1}{T}\log\nu_{f}( Y_{T}(\gamma,Q,M))\geq I^{f,g}_{\infty}(\gamma),$$
concluding the proof of the lower bound. \end{proof}

To conclude the proof we need to consider the case where $f$ is SPR. For part (a)  if $s_\infty(f)<1$ and $\gamma\in (\int g \,\d\nu_f,\beta(g))$ then $I^{f,g}(\gamma)<0$ follows from Corollary \ref{awayfromzero}(a). For part (b) if $\gamma<\beta_\infty(g)$ then $I^{f,g}_\infty(\gamma)<0$ is a consequence of Corollary \ref{awayfromzero}(b). Thus the proof of Theorem \ref{thm:ldp} is complete.

\subsection{Further consequences}

We start by proving Corollary \ref{cor:version2} from the introduction.

\begin{proof}
We prove the  two claims separately. 
\begin{enumerate}
\item 
This follows from applying Theorem \ref{thm:ldp}(a) both to $g$ with $\gamma=\int g\,\d\nu_{f}+\epsilon$ and $-g$ with $\gamma=-\int g\,\d\nu_f+\epsilon$.
\item 
This is immediate from taking $\gamma$ sufficiently small in  Theorem \ref{thm:ldp}(b) and using the fact that 
since $f$ is SPR we have $I_{\infty}^{f,g}(\gamma)\leq P_{\infty}^{\Phi}(g)<0$.

\end{enumerate}
\end{proof}

We also get some consequences in the discrete time case, where we set
$K_m=\bigcup_{i=1}^{m}[i].$

\begin{corollary}\label{discrete}
Let $(\Sigma,\sigma)$ be topologically mixing, $\phi:\Sigma\to\R$ be a function of summable variations with zero pressure and equilibrium state $\mu_{\phi}$. Moreover, assume that $\sup_{x\in[b]}|\phi(x)|< \infty$, for all $b\in\N$.
Let $\psi:\Sigma\to\R$ be a bounded function of summable variations. We have for all $m\in\N$
\begin{enumerate}
\item\label{corpart1}
For $\gamma<\beta(\psi)$ we have that
$$\lim_{n\to\infty}n^{-1}\log\mu_{\phi}\{x\in K_m\cap\sigma^{-n}(K_m):S_n\psi(x)> n\gamma\}=I^{\phi, \psi}(\gamma).$$
\item 
Moreover, if $\phi$ is SPR and $\gamma\in (\int\psi\,\d\mu_{\phi},\beta(\psi))$ then $I^{\phi, \psi}(\gamma)<0$.
\end{enumerate}
\end{corollary}

\begin{proof}
We prove the two claims.
\begin{enumerate}
\item 
We set $\tau:\Sigma\to [0,\infty)$ to be constantly equal to $1$, $f:Y\to\R$ to satisfy $f((\omega,t))=\phi(\omega)$ for all $(\omega,t)\in Y$ and $g((\omega,t))=\psi(\omega)$ for all $(\omega,t)\in Y$. It now follows from Theorem \ref{thm:ldp} that there exists $m\in\N$ such that For $\gamma<\beta(\psi)$ we have that
$$\lim_{n\to\infty}n^{-1}\log\mu_{\phi}\{x\in K_m\cap\sigma^{-n}(K_m):S_n\psi(x)> n\gamma\}=I^{\phi, \psi}(\gamma).$$
To complete the proof we need to show the result for $K_1$. To do this note that there exists $N\in\N$ such that for all $1\leq i\leq m$  there exists admissible words $(\omega_1,\ldots,\omega_N)$, $(\tau_1,\ldots,\tau_N)$ such that $\omega_1=1$, $\tau_N=1$, $\omega_N=i$ and $\tau_1=i$. Thus
$$\mu_{\phi}\{x\in K_1\cap\sigma^{-(n+2N)}(K_1):S_{n+2N}\psi(x)> n\gamma-2N\norm
{\psi}_{\infty}\}\geq C\mu_{\phi}\{x\in K_m\cap\sigma^{-n}(K_m):S_n\psi(x)> n\gamma\}$$

\item By the proof of Theorem \ref{RATEVARIATION}, there exists $q_0 \geq 0$ such that $P(\phi+q_0(\psi-\gamma))=I^{\phi,\psi}(\gamma)$. Since $\gamma\in \left(\int \psi \, \d\mu_{\phi},\beta(\psi)\right)$ we have that the right derivative of $P(\phi+q(\psi-\gamma))$ at $q=0$ is negative (see \cite[Theorem 3.2]{rs}) and since the asymptotic slope is positive we have $q_0>0$.
Since  $P(\phi)=0$ we obtain  $I^{\phi,\psi}(\gamma)=P(\phi+q_0(\psi-\gamma))<0$. 
\end{enumerate}    
\end{proof}

\begin{corollary}
Let $(\Sigma,\sigma)$ be topologically mixing and satisfy the BIP property, $\phi:\Sigma\to\R$ be a function of summable variations with zero pressure and equilibrium state $\mu_{\phi}$. Moreover, assume that $\sup_{x\in[b]}|\phi(x)|< \infty$ for all $b\in\N$.
Let $\psi:\Sigma\to\R$ be a bounded function of summable variations. We have for all $m\in\N$ and $\gamma<\beta(\psi)$  that 
$$\lim_{n\to\infty}n^{-1}\log\mu_{\phi}\{x\in \Sigma:S_n\psi(x)> n\gamma\}=I^{\phi, \psi}(\gamma).$$
\end{corollary}
\begin{proof}
Since  $(\Sigma,\sigma)$ is topologically mixing and satisfy the big images property this means there exists $N\in N$ such that for all $m\in\N$ and all $n\geq N$ there exist admissible words $\omega_1,\ldots,\omega_l$, $\tau_1,\ldots,\tau_l$ such that $\omega_1=1$, $\tau_1=1$, $\omega_n=m$ and $\tau_n=m$. The result now can be proved using exactly the method as used in part (\ref{corpart1}) of Corollary \ref{discrete}.    \end{proof}

\section{Applications and Examples} \label{sec:ex}

\subsection{ Pressure gap}

In 2001, Kifer, Peres, and Weiss \cite{kpw} established a \emph{dimension gap} for 
Bernoulli measures of the Gauss map. Specifically, they proved the existence of a 
constant  $c \in (0,1)$ such that the Hausdorff dimension of any Bernoulli measure is 
bounded above by $1-c$. Several refinements and alternative proofs of this result have 
since been developed \cite{bj, j, p}. The original proof was derived as a consequence of a large deviation result. This same line of reasoning yields a similar bound for the pressure (see Theorem~\ref{pressuregap}). Furthermore, this approach can be extended to obtain uniform bounds for the free energy of invariant measures for suspension flows that project onto $k$-Markov measures (see Corollary~\ref{coro:gap}). Interestingly, the Hausdorff dimension of an ergodic invariant measure $\mu$ for the Gauss map coincides with the entropy of the associated measure $\nu$ in the suspension flow. In this context, the base dynamics is the Gauss map $G$, and the roof function is given by $\log |G'|$. 

Recall that a  measure $\mu\in \M_\sigma$ is $k$-step Markov if for every $m\in \N$, $b\in \N^m$, $c\in \N$  and $a\in\N^k$ such that $\mu([b,a])>0$, then \begin{align}\label{z} \frac{\mu([b,a, c])}{\mu([ b,a])}=\frac{\mu([a,c])}{\mu([ a])}.\end{align} 

\begin{theorem}\label{pressuregap}
Let $(\Sigma,\sigma)$ be topologically mixing and let $\phi:\Sigma\to\R$ be of summable variations, $\sup_{x\in[b]}|\phi(x)|<\infty$ for every $b\in\N$, and such that $P_{\infty}(\phi)<P(\phi)=0$. Furthermore, assume that $\phi$ admits an equilibrium state $\mu_\phi$. Then, either $\mu_\phi$ is a Markov measure or for any $k\in\N$
$$\sup \left\{h_{\mu}(\sigma)+\int\phi\, \d\mu:\mu\text{ is $k$-step Markov and} \int\phi\, \d\mu>-\infty \right\}<P(\phi)=0.$$
\end{theorem}
\begin{proof}
Suppose that $\mu_{\phi}$ is not a Markov measure. Note that there exist $m\in \N$, $b\in \N^m$, $c\in \N$  and $a\in\N^k$ such that $\mu_\phi([b,a])>0$, and
$$\delta:=|\mu_\phi([a])\mu_\phi([b,a,c])-\mu_\phi([b,a])\mu_\phi([a,c])|>0.$$ On the other hand, if $\mu$ is a $k$-step Markov measure, then $\mu([a])\mu([b,a,c])=\mu([b,a])\mu([a,c]).$ Using exactly the same argument as in \cite[Corollary 2.3]{kpw} define $$\psi_1=\chi_{[a]},\psi_2=\chi_{[b,a,c]}, \psi_3=\chi_{[b,a]},\text{ and }\psi_4=\chi_{[a,c]}$$ Observe that if $\mu$ is a Markov measure, then $\max_{1\le i\le 4}|\int \psi_i\d\mu-\int\psi_i \,\d\mu_\phi|\ge \delta/4$. Thus we have
$$\sup \left\{h_{\mu}(\sigma)+\int\phi\, \d\mu:\mu\text{ is Markov} \right\}\leq\max_{1\leq i\leq 4}
\sup\left\{h_{\mu}(\sigma)+\int\phi\, \d\mu:\left|\int \psi_i \d\mu-\int\psi_i \,\d\mu_\phi\right|\ge \delta/4 \right\}<0,$$
where we have used Corollary \ref{discrete}.
\end{proof}

This has a  corollary in terms of a entropy gap for suspension flows.

\begin{corollary} \label{coro:gap}
Let $(\Sigma,\sigma)$ be a topologically mixing countable Markov shift and $\tau:\Sigma\to\R$ a roof function bounded away from zero, with summable variations and $\sup_{x\in[b]}\tau(x)<\infty$ for every $b\in\N$. Let $(Y,\Phi)$ be the associated suspension flow which we assume has finite entropy. Suppose that $h_{\infty}(\Phi)<h(\Phi)$. Then, the unique measure of maximal entropy for $\Phi$ either projects to a Markov measure on $\Sigma$ or for any $k\in\N$
$$\sup \left\{h_{\nu}(\Phi):\nu=\frac{\mu\times\text{Leb}|_{Y}}{\mu\times\text{Leb}(Y)}\text{ and }\mu\text{ is $k$-Markov } \right\}<h(\Phi),$$
\end{corollary}

\begin{proof}
Since $h_{\infty}(\Phi)<h(\Phi)$ the suspension flow is SPR and by \cite[Theorem 8.1]{v} (see also \cite[Lemma 4.12]{irv2} and Theorem \ref{thm:spr_deri}) it has a unique measure of maximal entropy $\nu$ with $\nu_m=(\mu\times\text{Leb}|_{Y})/(\mu\times\text{Leb}(Y))$ for a $\sigma$-invariant probability measure $\mu$. We also have that
$$\inf\{t:P(-t\tau)\leq 0\}=h(\Phi)=h_{\nu_m}(\Phi)=\frac{h_{\mu_m}(\sigma)}{\int \tau \,\d\mu_m}.$$
Thus, $P(-h(\Phi)\tau)=0$ with $\mu_m$ being the equilibrium state for $-h(\Phi)\tau$ (see \cite[Corollary 3.6]{ijt}). 
Furthermore, we have that
$$h_{\infty}(\Phi)=\inf\{t:P_{\infty}(-t\tau)\leq 0\}<h(\Phi).$$
Since $\tau$ is bounded away from $0$ this means that $P_{\infty}(-h(\Phi)\tau))<0$.
Thus by Theorem \ref{pressuregap} it follows that
$$s_k:=\sup \left\{h_{\mu}(\sigma)-h(\Phi)\int\tau\, \d\mu:\mu\text{ is $k$-step Markov and} \int\tau\, \d\mu<\infty \right\}<P(\phi)=0.$$
Now consider
$$h_k=\sup \left\{h_{\nu}(\Phi):\nu=\frac{\mu\times\text{Leb}|_{Y}}{\mu\times\text{Leb}(Y)}\text{ and }\mu\text{ is $k$-Markov } \right\}.$$
Let $h_{\infty}(\Phi)<t<h(\Phi)$ and set 
$$r_t=\sup\left\{\int\tau\, \d\mu:\frac{\mu\times\text{Leb}|_{Y}}{\mu\times\text{Leb}(Y)}\text{ and }\mu\text{ is $k$-Markov  and }\frac{h_{\mu}(\sigma)}{\int\tau\, d\mu}\geq t\right\}.$$
Since $t>h_{\infty}(\Phi)$ we must have $r_t<\infty$; recall that if $\lim_{n\to\infty}\int \tau d\mu_n=\infty$, then $\frac{h_{\mu}(\sigma)}{\int\tau\, d\mu}\le h_\infty(\Phi)$ (see Remark \ref{rem:topsus}). Thus, for $\mu$ a $k$-step Markov measure on $\Sigma$ and $\nu=\frac{\mu\times\text{Leb}|_{Y}}{\mu\times\text{Leb}(Y)}$ we have  two cases. The first case is that $\frac{h_{\mu}(\sigma)}{\int\tau\,\d\mu}\leq t$ and so $h_\nu(\Phi)\leq t$. The second case is that
$\int\tau\, \d\mu\leq r_t$ and so
$$s_k\geq h_{\mu}(\sigma)-h(\Phi)\int\tau\, \d\mu$$
which since $s_k<0$ gives that
$$\frac{h_{\mu}(\sigma)}{\int\tau\, \d\mu}\leq\frac{s_k}{r_t}+h(\Phi).$$
Putting this together we obtain that
$$h_\nu(\Phi)\leq\max\left\{t, \frac{s_k}{r_t}+h(\Phi)\right\},$$
which since $s_k<0$ yields the result.\end{proof}

\subsection{LDP for suspension flows over the full shift}

In Theorem \ref{thm:ldp} we have the condition that $(x,t)\in R_T(m)$ whereas standard large deviation results would consider sets of the form
\begin{equation}\label{standardld}
E^g_{T}(\gamma):=\left\{(x,t)\in Y:\frac{1}{T}\int_{0}^{T} g(\Phi_s((x,t))\, \d s>\gamma\right\}.
\end{equation}
We look at situations when the results in Theorem \ref{thm:ldp} can be transferred to the sets above.  
We start with a simple observation
\begin{lemma}\label{distortion1}
Suppose that $\Delta_f,\Delta_g$ and $\tau$ have finite first variation and summable variations. Then, there exists $C>0$ such that for all $n\in\N$, $\omega\in\Sigma$ and $x,y\in [\omega_1,\ldots,w_n]$
$$\max\{|S_n\Delta_{f}(x)-S_n\Delta_{f}(y)|,|S_n\Delta_{g}(x)-S_n\Delta_{g}(y)|,|S_n\tau(x)-S_n\tau(y)|\}\leq C.$$
\end{lemma}

Throughout this section we will assume that $\Sigma=\N^{\N}$ is the full shift and let $\tau:\Sigma\to [0,\infty)$ be bounded away from $0$, has summable variations and finite first variation. Let $(Y,\Phi)$ be the associated suspension flow. It follows by Theorem \ref{thm:spr_deri} that if $f:Y\to\R$ is SPR and $s_\infty(f)<1$, then $f$ has a unique equilibrium state $\nu_f=\frac{(\mu_f\times\text{Leb})|_Y}{\int \tau d\mu_f}$ where $\mu_{f}$ is the Gibbs state on $\Sigma$ for the function $\Delta_{f}-P^{\Phi}(f)\tau$. Consider the set
$$A_r=\{x\in\Sigma:\tau(x)\geq r\}.$$

\begin{theorem}\label{cor:full}
Let $f:Y\to\R$ be a SPR function with $P^{\Phi}(f)=0$ and $s_\infty(f)<1$. Let $\nu_{f}$ be the equilibrium state and $\mu_{f}$ as above. Let $g:Y\to\R$ be bounded and assume that $\Delta_g$ has summable variations and finite first variation.

\begin{enumerate}
    \item If $\lim_{r\to\infty}\frac{\log\mu_{f}(A_r)}{r}=-\infty$, then for $\gamma<\beta(g)$ we have that
$$\lim_{T\to\infty}\frac{1}{T}\log \nu_f\big(E^g_{T}(\gamma)\big)
=I^{f,g}(\gamma).$$

\item 
If $\limsup_{r\to\infty}\frac{\log\mu_{f}(A_r)}{r}<0,
$ then for $\gamma\in (\int g\d\nu_{f},\beta(g))$ we have that
$$\limsup_{T\to\infty}\frac{1}{T}\log \nu_f\big(E^g_{T}(\gamma)\big)<0.$$

\end{enumerate}
\end{theorem}

\begin{proof}
We let $\gamma<\beta(g)$ and  $t_0=\inf\{\tau(\omega):\omega\in\Sigma\}.$ For each $T>0$ define $N(T)=\lceil\frac{T}{t_0}\rceil+2$. For $1\leq n\leq N(T)$ we set
\begin{eqnarray*}
Y(\gamma,T,n)=\bigg\{[\omega_1,\ldots,\omega_n]:\exists (x,t)\in ([\omega_1,\ldots,\omega_n]\times\mathbb{R})\cap Y_0\text{ with }\int_{0}^{T}g(\Phi_s(x,t)) \, \d s>T\gamma\\
\text{ and }S_{n-1}\tau(x)\le t+T< S_{n}\tau(x)\bigg\}.
\end{eqnarray*}
We now fix $\epsilon>0$ and
$$C(\gamma,T, \epsilon,n)=\{[\omega_1,\ldots,\omega_n]:\sup_{x\in [\omega_1]}\tau(x), \sup_{x\in[\omega_n]}\tau(x)\leq\epsilon T\}.$$
We will initially focus on the collection of good cylinders $G(\gamma,T, \epsilon,n):=Y(\gamma,T,n)\cap C(\gamma,T, \epsilon,n)$.

We take $C>0$ from Lemma \ref{distortion1} and note that this means for every $W\in G(\gamma,T, \epsilon,n)$ and every $y\in W$, we  have
$$S_n\tau(y)\in (T-C,T+C+2\epsilon T).$$
Define $t_1=\inf_{x\in[1]}\tau(x)$,   $t_2=\sup_{x\in[1]}\tau(x)$ and choose the smallest $k$ such that $kt_1>2C+2\epsilon T+t_2$. Thus, if we consider a cylinder set of the form $\widehat{W}:=[1,\omega_1,\ldots,\omega_n,(1)^k]$ where $W=[\omega_1,\ldots,\omega_n]\in G(\gamma,T, \epsilon,n)$ then for $T_1=T+C+2\epsilon T+t_2$ we will have  
$$(\widehat{W}\times\R)\cap Y_0\subset R_{T_1}(1).$$
Moreover, for any $(x,t)\in (\widehat{W}\times\R)\cap Y_0 $ we will have
$$\int_{0}^{T_1} g(\Phi_s(x,t))\, \d s>T\gamma-C_1-C_2\epsilon T$$
for some constants $C_1,C_2$ independent of $T,\epsilon, n$ and $W$. Here we used that $\Delta_g$ has summable variation and finite first variation, together with boundedness of $g$. It follows that
$$\frac{1}{T_1}\int_0^{T_1} g(\Phi_s(x,t))\, \d s>\gamma\frac{T}{T_1}-\frac{C_1}{T_1}-C_2\epsilon \frac{T}{T_1}=:\gamma_{\epsilon}(T).$$
Note that $\gamma_{\epsilon}(T)\to \gamma_\epsilon:=\frac{\gamma-C_2\epsilon}{1+2\epsilon}$ as $T\to\infty$ and $\gamma_\epsilon\to\gamma$ as $\epsilon\to 0$. Fix $\rho>0$. Choose $T$ sufficiently large that $\gamma_\epsilon(T)>\gamma_\epsilon-\rho$. Fix $\delta>0$, using Theorem \ref{thm:ldp} (noticing that the upper bound holds for any $m$) and the Gibbs property, for $T$ sufficiently large we have
\begin{eqnarray*}
\exp\left(T_1(I^{f,g}\left(\gamma_\epsilon-\rho\right)+\delta)\right)&\geq &\nu_{f}\left(\left\{(x,t)\in R_{T_1}(1):\frac{1}{T_1}\int g(\Phi_s(x,t))\, \d s>\gamma_\epsilon-\rho\right\}\right)\\
&\geq&\nu_{f}\left(\left\{(x,t)\in R_{T_1}(1):\frac{1}{T_1}\int g(\Phi_s(x,t))\, \d s>\gamma_\epsilon(T)\right\}\right)\\
&\geq&\sum_{W\in G(\gamma,T, \epsilon,n)}\nu_{f}((\widehat{W}\times\R)\cap Y_0
)\\
&=&\frac{1}{\int\tau d\mu_f}\sum_{W\in G(\gamma,T, \epsilon,n)}\int_{\widehat{W}}\tau d\mu_f\\
&\geq&\frac{t_1}{\int\tau d\mu_f}\sum_{W\in G(\gamma,T, \epsilon,n)}\mu_f(\widehat{W})\\
&\geq&C'\mu_{f}([(1)^k])\sum_{W\in G(\gamma,T, \epsilon,n)}\mu_{f}(W),
\end{eqnarray*}
where $C'$ is independent of $T,n,k$ and $W$. By the Gibbs property there exist $C_0>0$ and $\eta\in(0,1)$ such that $\mu_f([(1)^k])\ge C_0\eta^k$ and therefore
\begin{equation*}
\log \mu_{f}([(1)]^k)\geq \log C_0 +k\log\eta\geq \log C_0+\left(\frac{2C+2\epsilon T+t_2}{t_1}+1\right)\log\eta,
\end{equation*}
where we used that $k=\lfloor\frac{2C+2\epsilon T+t_2}{t_1}\rfloor+1.$

\begin{eqnarray*}
\frac{1}{T}\log\left( \max_{n}\sum_{W\in G(\gamma,T, \epsilon,n)}\mu_{f}(W)\right)&\leq& \frac{T_1}{T} \big(I^{f,g}(\gamma_\epsilon-\rho)+\delta\big)-\frac{2\epsilon\log \eta}{t_1}-\frac{D+\log (C_0C')}{T}
\end{eqnarray*}
for some $D$. In other words
\begin{eqnarray}\label{ineq:good}
\limsup_{T\to\infty}\frac{1}{T}\log\left(\max_{n} \sum_{W\in G(\gamma,T, \epsilon,n)}\mu_{f}(W)\right)&\leq& (1+2\epsilon)(I^{f,g}(\gamma_\epsilon-\rho)+\delta)+O(\epsilon)
\end{eqnarray}

We now go to the main set that we are estimating. Let $V_1$ be the finite first variation of $\tau$. Note that if $W=[\omega_1,\ldots,\omega_n]\in Y(\gamma,T,n)\setminus G(\gamma,T,\epsilon, n)$, then $\sup_{x\in [\omega_1]}\tau(x)>\epsilon T$, or $\sup_{x\in [\omega_n]}\tau(x)>\epsilon T,$ and therefore $\tau(x)\ge\epsilon T-V_1$ on $[\omega_1]$ or on $[\omega_n]$. In particular, for $T>0$ we get
\begin{eqnarray*}\label{calc}
\nu_{f}(E^g_{T}(\gamma))&\leq&\sum_{n=1}^{N(T)}\sum_{W\in Y(\gamma,T,n)}\nu_{f}((W\times\R)\cap Y_0)\nonumber\\
&\leq& \sum_{n=1}^{N(T)}\sum_{W\in G(\gamma,T, \epsilon,n)}
\nu_{f}((W\times\R)\cap Y_0)+\nu_f(\{(x,t)\in Y:\tau(x)\ge\epsilon T-V_1\})\nonumber\\
&+&\nu_f(\Phi_T^{-1}\{(x,t)\in Y:\tau(x)\ge \epsilon T-V_1\})\nonumber\\
&=& \sum_{n=1}^{N(T)}\sum_{W\in G(\gamma,T, \epsilon,n)}
\nu_{f}((W\times\R)\cap Y_0)+2\nu_f(\{(x,t)\in Y:\tau(x)\ge\epsilon T-V_1\})\nonumber\\
\end{eqnarray*}
Here we used that $\nu_f(\Phi_T^{-1}\{(x,t)\in Y:\tau(x)\ge\epsilon T-V_1\})=\nu_f(\{(x,t)\in Y:\tau(x)\ge\epsilon T-V_1\})$ by invariance. Note that if $W\in G(\gamma,T,\epsilon, n)$, then
$$\nu_f((W\times\R)\cap Y_0)=\frac{\int_{W}\tau d\mu_f}{\int \tau d\mu_f}\le (\epsilon T) \frac{ \mu_f(W) }{\int \tau d\mu_f},$$ since $\tau(x)\le \epsilon T$, for every $x\in W$. We obtain that 
\begin{align}\label{ineq:final}
    \nu_{f}(E^g_{T}(\gamma))&\le \sum_{n=1}^{N(T)}\frac{\epsilon T}{\int \tau d\mu_f}\sum_{W\in G(\gamma,T, \epsilon,n)}
\mu_{f}(W)+2\nu_f(\{(x,t)\in Y:\tau(x)\ge\epsilon T-V_1\})\nonumber\\
&\le N(T)\frac{\epsilon T}{\int \tau d\mu_f}\max_n\sum_{W\in G(\gamma,T, \epsilon,n)}
\mu_{f}(W)+2\nu_f(\{(x,t)\in Y:\tau(x)\ge\epsilon T-V_1\})
\end{align}
Note that $N(T)=O(T)$. The exponential growth of the left hand side is bounded by the maximum of the exponential growths of $\max_n\sum_{W\in G(\gamma,T, \epsilon,n)}
\mu_{f}(W)$ and $\nu_f(\{(x,t)\in Y:\tau(x)\ge\epsilon T-V_1\})$  

Note that $$\nu_f\big(\{(x,t)\in Y:\tau(x)\ge r\}\big)=\frac{\int_{A_{r}}\tau d\mu_f}{\int \tau d\mu_f}=\frac{r\mu_f(A_{r})+\int_{r}^\infty \mu_f(A_t)dt}{\int \tau d\mu_f}$$

Let us suppose that $\limsup_{r\to\infty}\frac{\log \mu_f(A_r)}{r}<-\alpha$ for some $\alpha>0$. In other words, for $r$ sufficiently large we have that 
$\mu_f(A_r)<e^{-\alpha r}$. In particular, $\int_r^\infty \mu_f(A_t)dt<\int_r^\infty e^{-\alpha t}dt=\frac{1}{\alpha }e^{-\alpha r}$, 
and therefore 
\begin{align*}\label{ineq:bad}
    r\mu_f(A_{r})+\int_{r}^\infty \mu_f(A_t)dt \le r e^{-\alpha r}+\frac{e^{-\alpha r}}{\alpha}.
\end{align*}
Considering $r=\epsilon T-V_1$ large enough, we get that
\begin{eqnarray}\label{ineq:alpha}
    \limsup_{T\to\infty}\frac{1}{T}\log \nu_f(\{(x,t)\in Y:\tau(x)\ge\epsilon T-V_1\})\le -\alpha\epsilon.
\end{eqnarray}
In part (a) we have that 
 $\lim_{r\to\infty}\frac{\log \mu_f(A_r)}{r}=-\infty$, and therefore $\alpha$ can be arbitrarily large. Then 
 \begin{eqnarray}\label{ineq:inf}
     \limsup_{T\to\infty}\frac{1}{T}\log \nu_f(\{(x,t)\in Y:\tau(x)\ge\epsilon T-V_1\})=-\infty. 
 \end{eqnarray}
 In this case it follows by inequalities (\ref{ineq:good}),  (\ref{ineq:final}) and (\ref{ineq:inf}) that
$$\limsup_{T\to\infty}\frac{1}{T}\log \nu_f( E^g_{T}(\gamma))
\le (1+2\epsilon)(I^{f,g}(\gamma_\epsilon-\rho)+\delta)+O(\epsilon).$$
Since $\delta>0$ and $\rho>0$ are arbitrary, sending $\epsilon \to 0$ we obtain the desired upper bound (see Lemma \ref{continuity}). The lower bound follows directly by Theorem \ref{thm:ldp} since we estimate a restricted subset. 

In case (b) we can only conclude that 
$$\limsup_{T\to\infty}\frac{1}{T}\log \nu_f(E^g_{T}(\gamma))
\le \max\big\{(1+2\epsilon)(I^{f,g}(\gamma_\epsilon-\rho)+\delta)+O(\epsilon), -\alpha\epsilon\big\}.$$
It follows from Corollary \ref{awayfromzero}(a) that under our assumptions $I^{f,g}(\gamma)<0$, and therefore for $\delta>0$, $\rho>0$ and $\epsilon>0$ small we obtain
$$\limsup_{T\to\infty}\frac{1}{T}\log \nu_f\big(E^g_{T}(\gamma)\big)<0,$$
as desired.
\end{proof}

\begin{example} Let $(\Sigma, \sigma)$ be the full shift in a countable alphabet. Let $\tau( x)=\sqrt{x_1}$ and $\Delta_{f}(x)=\log\left(\frac{C}{x_1^2}\right)$ where $C=\frac{6}{\pi^2}$. In this setting, we have $P^{\Phi}(f)=P^{\Phi}_{\infty}(f)=0$ and so $\phi$ is not SPR. 
It can be seen that, in general, we do not have exponential decay for our sets, since
$$\mu_{f}([1,n,1]\times\R\cap Y_0
)=\frac{C^3}{n^2}$$
and so, for example, if you consider $g:Y\to\R$ defined by 
$$g(\omega,t)=\left\{\begin{array}{ccc}1&\text{ if }&\omega_0\text{ is odd}\\
2&\text{ if }&\omega_0\text{ is even}\end{array}\right.$$
In this case, we have $\int g\,\d\nu_{f}<2$ and if we take $\gamma\in (\int g\,\d \nu_{f},2)$
we will have for all $m\in\N$,
$$\lim_{T\to\infty}\frac{1}{T}\log\left(\nu_{f}\left(\left\{(x,t)\in R_T(m):\frac{1}{T}\int_{0}^{T}g(\Phi_s((x,t))\, \d s>\gamma\right\}\right)\right)=0$$
and so $I^{f,g}(\gamma)=0$ for all $\gamma\in (\int g \, \d\nu_{f},2)$.
\end{example}

\begin{example}
Let $(\Sigma, \sigma)$ be the full shift. Let $\tau(x)=\log x_1+1$ and $\Delta_{f}(x)=\log\left(\frac{C}{x_1^2}\right)$  where $C=\frac{6}{\pi^2}$. In this setting, we have $P^{\Phi}(f)=0>P^{\Phi}_{\infty}(f)$ and so $f$ is SPR. Thus, our results can be applied to any bounded $g:Y\to\R$ where $\Delta_{g}$ has bounded variations. Moreover, we have that $I^{f,g}(\gamma)<0$ for all $\gamma\in (\int g\, \d\nu_{f},\beta(\gamma))$. In this case the conclusion of Theorem \ref{cor:full}(a) does not hold, but part (b) does hold. 
\end{example}

\begin{example}
Let $(\Sigma, \sigma)$ be the full shift. Let $\tau(x)=\log x_1+1$ and $\Delta_{f}(x)=-x_1\log 2$. In this setting, we have $P^{\Phi}(f)=0>P^{\Phi}_{\infty}(f)$ and so $f$ is SPR. In this case Theorem \ref{cor:full}(a) holds for any regular and bounded $g:Y\to\R$ where $\Delta_{g}$ has summable variations. Again we will have $I^{f,g}(\gamma)<0$ for all $\gamma\in (\int g\, \d\nu_{f},\beta(\gamma))$.
\end{example}

\subsection{Further examples}
We now turn our attention to where the base map is not the full shift.

\begin{example}[The renewal shift]
Consider the
transition matrix $A=(a_{ij})_{i,j \in \N}$ with $a_{1,1}= a_{1,n}=
a_{n,n-1}=1$ for each $n \ge 2$ and with all other entries equal to
zero. The \emph{renewal shift} is the Markov shift
$(\Sigma_R, \sigma)$ associated with the matrix $A$. This system has topological entropy equal to 
$\log 2$, while its entropy at infinity is zero.
The corresponding thermodynamic formalism is well understood; see \cite[Section 4]{sa2}.
Several interval maps, in particular the Manneville–Pomeau map, can be modeled by this shift.
Suspension flows over the renewal shift have also been studied, for example in \cite[Section 6]{ijt}.
Consider the locally constant roof function $\tau:\Sigma_R \to \R$ $\tau|[n]=t_n$ and $0<c< \tau < C$ with the additional property that $\lim_{n \to \infty} t_n=t_{\infty}$. The associated suspension flow $(Y, \Phi)$ has finite entropy. Let $f:Y \to \R$ be an SPR function with equilibrium measure $\nu_{f}$, such that $\Delta_{f}$ is locally constant with $\Delta_{f}|[n]=a_n$ and  $\lim_{n \to \infty} a_n=a_{\infty}$. 
Consider also the function $g:Y \to \R$  such that $\Delta_{g}$ is locally constant with $\Delta_{g}|[n]=b_n$. If  $\lim_{n \to \infty} b_n=b_{\infty}$ then $\alpha_{\infty}(g)=\beta_{\infty}(g)=\frac{b_{\infty}}{t_{\infty}}$. Theorem \ref{thm:ldp} part $(b)$ applies for every $\gamma<\beta_\infty(g)$. In this case,
\[
I^{f,g}_\infty(\gamma)=\frac{a_\infty}{t_\infty} \quad\text{for }\gamma\leq\beta_\infty(g),
\]
and
\[
I^{f,g}_\infty(\gamma)=-\infty \quad\text{for }\gamma>\beta_\infty(g).
\]
For an arbitrary function $g:Y \to \R$ such that $\Delta_{g}$ is of summable variations, we have $q \mapsto P^{\Phi}_{\infty}(q g)= q \beta_{\infty}(g)$ for $q >0$.
   In this case we have an explicit expression for the rate function at infinity: 
\[
 I^{f,g}_\infty(\gamma)=\inf_{q\ge 0} P^\Phi_\infty(f+q(g-\gamma))=\inf_{q\ge 0} \left(\frac{a_{\infty}}{t_{\infty}} + q \left(\beta_{\infty}(\psi) -\gamma \right)  \right).
 \]
 Thus, if $\gamma \leq \beta_{\infty}(g)$ then $I_\infty(\gamma)= \frac{a_{\infty}}{t_{\infty}}$. Othrewise, $I_\infty(\gamma)=-\infty$.

Note that if $\tau$ is the constant function equal to one we obtain large deviation estimates for the (discrete time) renewal shift.
\end{example}

\begin{example}[Random walk] \label{rw2}
  
Consider the Random Walk  CMS $(\Sigma, \sigma)$ defined in Example \ref{rw}. This is a topologically mixing CMS of finite entropy that does not have a measure of maximal entropy.
Consider the suspension flow $(Y, \Phi)$ over $(\Sigma, \sigma)$ with roof function $\tau$ constant equal to one.
Let $f:Y \to  \R$ be a zero pressure  SPR potential with equilibrium measure $\nu_{f}$ such that $\Delta_{f}$ is of summable variations.  
Let $a, b \in \R$ be two negative numbers, as in Example \ref{rw}, let $g:\Sigma \to \R$ be defined by
\begin{equation*}
g(x)=
\begin{cases}
a & \text{ if  } x_1 \text{ is odd};\\
b & \text{ if  } x_1 \text{ is even}
    \end{cases}
\end{equation*}
For every $t \in \R$ we have $P(tg)=P_{\infty}(t g)$. Assume that $a<b$, we then have that $\alpha(g)=\alpha_{\infty}(g)=a$ and 
$\beta(g)=\beta_{\infty}(g)=b$. Thus, in this example, the range of values for which the conclusions of Theorem \ref{thm:ldp} are valid coincides in statements $(a)$ and $(b)$.
\end{example}

\begin{example}[Loop systems]
Ruette \cite{ru1} studied \emph{loop systems} and constructed examples exhibiting different  dynamical properties.
Consider a countable Markov shift defined by a directed graph with a single central vertex, where $a(n)$ denotes the number of distinct loops of length $n$ based at this vertex. Let $F(z) = \sum_{n=1}^{\infty} a(n) z^n$ and $R$ be its radius of convergence. As observed by Ruette, the recurrence properties of the system are entirely determined by the behavior of $F(z)$ evaluated at $R$: the system is \emph{strongly positive recurrent} (SPR) if $F(R) > 1$, and it is \emph{positive recurrent but not SPR} if $F(R) = 1$ with $F'(R) < \infty$. Different choices of the sequence $a(n)$ yield a wide range of dynamical properties. We highlight two specific cases.
Following Ruette \cite[Example 2.9]{ru1}, we can construct a boundary case by taking $a(k^2)=2^{k^2-k}$ for $k \geq 1$, $a(n)=0$ otherwise. Here, $R=1/2$. We have $F(1/2) = \sum_{k=1}^{\infty} 2^{-k} = 1$ and $F'(1/2) = \sum_{k=1}^{\infty} k^2 2^{1-k} < \infty$. As constructed, the associated system $(\Sigma,\sigma)$ is positive recurrent, but not strongly positive recurrent. To obtain a system that is strongly positive recurrent and has finite entropy, one can choose  $a(n)=2^n$ for all $n \ge 1$. The radius of convergence remains $R=1/2$, but evaluating the generating function yields $F(1/2) = \sum_{n=1}^{\infty} 1 = \infty > 1$, guaranteeing the strict inequality required for the SPR property.

Now consider the suspension flow $(Y,\Phi)$ over an SPR system $(\Sigma,\sigma)$ (such as the one defined by $a(n)=n$) with constant roof function $\tau \equiv 1$. Let $f \colon Y \to \mathbb{R}$ be a zero-pressure strongly positive recurrent potential with equilibrium measure $\nu_{f}$, and assume that $\Delta_{f}$ has summable variations. Then the conclusions of Theorem~\ref{thm:ldp} hold for any bounded function $g \colon Y \to \mathbb{R}$ such that $\Delta_{g}$ has summable variations.
\end{example}

\end{document}